\documentclass[11pt]{amsart}
\usepackage{amsmath,amssymb,amscd,amsfonts,verbatim}
\usepackage[margin=1in]{geometry}
\usepackage{rotating}
\usepackage{color}

\newtheorem{prop}[equation]{Proposition}

\newtheorem{thm}[equation]{Theorem}
\newtheorem{cor}[equation]{Corollary}
\newtheorem{lem}[equation]{Lemma}

\newtheorem{assumptions}[equation]{Assumptions}

\newtheorem{cria}[equation]{Criteria}

\theoremstyle{definition}

\newtheorem{defn}[equation]{Definition}
\newtheorem{defns}[equation]{Definitions}
\newtheorem{rem}[equation]{Remark}
\newtheorem{rems}[equation]{Remarks}

\newtheorem{exa}[equation]{Example}

\numberwithin{equation}{section} 

\newcommand{\sands}{\mbox{$\quad\text{and}\quad$}}

\newcommand{\Hom}{\operatorname{Hom}}

\newcommand{\letbe}{\mathbin{\raisebox{.3pt}{:}\!=}}
\newcommand{\belet}{\mathbin{=\!\raisebox{.3pt}{:}}}

\newcommand{\2}{\mskip-2mu} 

\newcommand{\1}{\mskip-1mu} 
\newcommand{\?}{\mskip-1.5mu}
\newcommand{\om}{\mskip1mu}

\newcommand{\hot}{\mathbin{\widehat{\otimes}}}

\newcommand{\CPn}{\mbox{$\C P^n$}}

\newcommand{\col}{\operatorname{colim}}

\newcommand{\hocolim}{\operatorname{hocolim}}

\newcommand{\GL}{\mbox{\it GL\/}} 

\newcommand{\MU}{\mbox{$M\hspace{-.5pt}U$}}

\newcommand{\HwK}{\mbox{$H\?\wedge\? K$}}

\newcommand{\HwMU}{\mbox{$H\?\wedge\? MU$}}

\newcommand{\C}{\mathbb{C}}

\newcommand{\bP}{\mathbb{P}}
\newcommand{\Q}{\mathbb{Q}} 
\newcommand{\R}{\mathbb{R}}
\newcommand{\Z}{\mathbb{Z}} 

\newcommand{\brpnp}{\mbox{$\R_{\scriptscriptstyle\geqslant}^{n+1}$}}

\newcommand{\srf}[1]{\mbox{${\text{\it SR}^F}$}}
\newcommand{\Th}{\mbox{\it Th}}

\newcommand{\Pc}{\mbox{$\mathbb{P}(\chi)$}}
\newcommand{\Pd}{\mbox{$\bP(\chi')$}}

\newcommand{\Tc}{\mbox{$T(\chi)$}}
\newcommand{\acts}{\mathbin{\raisebox{-.5pt}{\reflectbox{\begin{sideways}
$\circlearrowleft$\end{sideways}}}}}
\newcommand{\scirc}{\mathbin{\raisebox{-.0pt}{$\scriptstyle\circ$}}}
\newcommand{\xvs}{\mbox{$X_\varSigma$}}

\newcommand{\vsy}{\mbox{$\varSigma_Y$}}
\newcommand{\vsc}{\mbox{$\varSigma_{\chi}$}}

\newcommand{\EV}{\mbox{$E\hspace{0pt}V$}}

\newcommand{\PP}{\mbox{$P\hspace{-1pt}P$}}

\newcommand{\KBV}{\mbox{$K\hspace{-1.8pt}_B\hspace{-.6pt}V$}}

\newcommand{\HR}{\mbox{$H\hspace{-1pt}R$}}

\newcommand{\MUV}{\mbox{$M\hspace{-.5pt}U\hspace{-.2pt}V$}}
\newcommand{\MUBV}{\mbox{$M\hspace{-.5pt}U\hspace{-1.8pt}_B\hspace{-.2pt}V$}}

\newcommand{\cat}[1]{\mbox{\sc #1}}
\newcommand{\cats}{\cat{cat}(\varSigma)}
\newcommand{\catsop}{\cat{cat}^{op}(\varSigma)}
\newcommand{\catgca}{\cat{gcalg}}
\newcommand{\scat}[1]{\mbox{\scriptsize{\sc #1}}}

\newcommand{\ra}{\rightarrow}
\newcommand{\lra}{\longrightarrow}
\newcommand{\llra}{\relbar\joinrel\hspace{-2pt}\lra}
\newcommand{\lllra}{\relbar\joinrel\hspace{-2pt}\llra}

\makeatletter
\newcommand{\Ddots}{\mathinner{\mkern1mu\raise\p@
\vbox{\kern7\p@\hbox{.}}\mkern2mu
\raise4\p@\hbox{.}\mkern2mu\raise7\p@\hbox{.}\mkern1mu}}
\makeatother

\DeclareMathOperator{\Fred}{Fred}

\begin{document}

\title[Equivariant {\em K}-theory and cobordism]
{The equivariant {\em K}-theory 
and cobordism rings of divisive weighted projective spaces}

\makeatletter
\def\@@and{and}
\renewcommand{\author@andify}{%
  \nxandlist{\unskip, }{\unskip{} \@@and~}{\unskip{} \@@and~}}
\makeatother

\author{Megumi Harada} \address{Department of Mathematics and
Statistics, McMaster University, 1280 Main Street West, Hamilton,
Ontario L8S4K1, Canada} \email{Megumi.Harada@math.mcmaster.ca}

\author{Tara S. Holm} \address{Department of Mathematics, Malott
Hall, Cornell University, Ithaca, New York 14853-4201, USA}
\email{tsh@math.cornell.edu}

\author{Nigel Ray} \address{School of Mathematics, University of
Manchester, Oxford Road, Manchester M13 9PL, UK}
\email{nigel.ray@manchester.ac.uk}

\author{Gareth Williams} \address{Department of Mathematics and 
Statistics, The Open University, Walton Hall, Milton Keynes MK7 6AA,
UK} \email{G.R.Williams@Open.ac.uk}

\keywords{Divisive weight vector, Equivariant K-theory, fan, piecewise
Laurent polynomial, weighted projective space}
\subjclass[2010]{Primary: 55N91; Secondary: 14M25, 57R18}



\begin{abstract}
We apply results of Harada, Holm and Henriques to prove that the
Atiyah-Segal equivariant complex $K$-theory ring of a divisive
weighted projective space (which is singular for non-trivial weights)
is isomorphic to the ring of integral piecewise Laurent polynomials on
the associated fan. Analogues of this description hold for other
complex-oriented equivariant cohomology theories, as we confirm in the
case of homotopical complex cobordism, which is the universal
example. We also prove that the Borel versions of the equivariant
$K$-theory and complex cobordism rings of more general singular toric
varieties, namely those whose integral cohomology is concentrated in
even dimensions, are isomorphic to rings of appropriate piecewise
formal power series. Finally, we confirm the corresponding
descriptions for any \emph{smooth}, compact, projective toric
variety, and rewrite them in a face ring context. In many cases our 
results agree with those of Vezzosi and Vistoli for algebraic 
$K$-theory, Anderson and Payne for operational $K$-theory, Krishna 
and Uma for algebraic cobordism, and Gonzalez and Karu for operational 
cobordism; as we proceed, we summarize the details of these 
coincidences.
\end{abstract}

\def\@@and{and}
\renewcommand{\andify}{%
  \nxandlist{\unskip, }{\unskip{} \@@and~}{\unskip{} \@@and~}}

\maketitle

%
%
%
%
%
%
%
%
%

\section{Introduction}\label{in}

Throughout this work $G$ is a compact Lie group, and $G\acts Y$ a
$G$-space; when $G$ is understood, we rewrite the latter as $Y$.
Our aim is to investigate the $G$-equivariant complex $K$-theory 
and complex cobordism rings of certain special families of $Y$, for which
$G$ is a torus $T^n$. Given the recent proliferation of $K$-theory 
and cobordism functors \cite{anpa:okt}, \cite{com:cce}, 
\cite{goka:bac}, \cite{krum:acr}, it is important to specify precisely
which we use, and to comment on their relationship with other versions
as we proceed. Our underlying philosophy is closest to algebraic 
topology and homotopy theory.

So far as $K$-theory is concerned, we focus mainly on the unreduced
Atiyah-Segal $G$-equivariant ring $K_G^*(Y)$ \cite{Seg68}, graded over
the integers for later convenience. If $Y$ is compact, then $K_G^0(Y)$
is constructed from equivalence classes of $G$-equivariant complex
vector bundles; otherwise, it is given by equivariant homotopy classes
$[Y,\Fred(\mathcal{H}_G)]_G$, where $\mathcal{H}_G$ is a Hilbert space
containing infinitely many copies of each irreducible representation
of $G$ \cite{AtiSeg04}.  For the $1$-point space $*$ with trivial
$G$-action, we write the \emph{coefficient ring} $K_G^*(*)$ as
$K_G^*$. It is isomorphic to $R(G)[z,z^{-1}]$, where $R(G)$ denotes
the complex representation ring of $G$, and realises $K_G^0$; the
\emph{Bott periodicity element} $z$ has cohomological dimension $-2$.
The equivariant projection $Y\to *$ induces the structure of a graded
$K_G^*$-algebra on $K^*_G(Y)$, for any $G\acts Y$.

For complex cobordism, our primary interest is tom Dieck's 
$G$-equivariant ring $MU_G^*(Y)$ \cite{td:bgm}, defined by equivariant
stable homotopy classes $[Y,\MU_G]_G$ of maps into the Thom spectrum 
$\MU_G$. Although the coefficient ring $MU_G^*$ remains undetermined, 
considerable information is available when $G=T^n$; Sinha \cite{sin:cce}, 
for example, has made extensive calculations, and 
solved the case $n=1$. The equivariant projection $Y\to *$ induces 
the structure of a graded $MU_G^*$-algebra on $MU^*_G(Y)$, for any 
$G\acts Y$. In fact $MU_G^*(-)$ is universal amongst unreduced 
complex-oriented $G$-equivariant cohomology theories, at least for 
abelian $G$ \cite{cogrkr:uec}. The most natural link between cobordism
and $K$-theory arises within this framework, and is provided by the 
equivariant Todd genus $td\colon MU^*_G(Y)\to K^*_G(Y)$ \cite{oko:cfi}.

Given the universality of $MU^*_G(-)$, we may follow the lead of 
\cite{HHH} and consider arbitrary complex-oriented cohomology 
theories $E^*_G(-)$, although we restrict attention to cases for 
which $G$ is a compact torus. We view the complex orientation as a 
choice of equivariant Thom class without further comment. In these 
situations, readers will lose little by interpreting $E$ as 
whichever of complex $K$-theory, complex cobordism, or integral 
cohomology takes their fancy.  When more detail is required, we shall treat 
$K$-theory first and cobordism second; both depend on the more familiar 
case $H^*_G(-)$ of integral cohomology, which we recall as necessary.

If $Y$ is an $n$--dimensional toric variety \cite{ful:itv}, it is
automatically endowed with an action of $T^n$. 
The problem of describing the ring $K_{T^n}^*(Y)$ 
has already been addressed, albeit indirectly, in the symplectic 
context \cite{HaL}. From the algebraic viewpoint, however, it is more 
natural to study algebraic vector bundles over $Y$, and to compute the 
corresponding algebraic $K$-theory. Influential contributions along 
these lines include Brion \cite{bri:spa}, Dupont \cite{dup:fpv}, 
Kaneyama \cite{kan:ctp}, Klyachko \cite{kly:ebt}, 
Liu and Yao \cite{liya:ste}, Morelli \cite{mor:ktt}, and 
Payne \cite{pay:tvb}, although the work of Vezzosi
and Vistoli \cite{vevi:hak} is closest to ours in spirit, and leads to
answers that are isomorphic to $K_{T^n}^*(Y)$ for all smooth $Y$. 
The results of \cite{vevi:hak} are also cited in the appendix to 
\cite{ros:ekt}, with appeal to arguments of Franz. The same arguments 
are likely to provide an alternative approach to our own results, but 
have yet to be fully documented. We emphasise our
insistence on integer coefficients, bearing in mind that several
of the above authors tensor their $K$-groups with $\Q$ or $\C$ 
throughout.

The techniques used by symplectic geometers apply to a wider class of
$T^n$-spaces, and are based on symplectic reduction, Morse theory
\cite{holm-matsumura}, and GKM graphs \cite{GKM}.  Algebraic
geometers, on the other hand, tend to restrict attention (at least 
over $\Z$) to smooth toric varieties $Y$, possibly non-compact, and
express their invariants in terms of the underlying fan $\vsy$. 

Our aim is to combine features of each viewpoint, and
describe $K_{T^n}^*(Y)$ in terms of $\vsy$ for a 
certain family of \emph{singular} toric varieties. Following 
\cite{bafrra:wps}, we refer to these as 
\emph{divisive weighted projective spaces}, and denote them by
$\Pc$. Recent work \cite[Theorem 1.2]{bafrnora:cwp} shows that 
\emph{any} weighted projective space is homotopy equivalent to 
one which is divisive, but such equivalences need not be equivariant, 
by the concluding remarks of \cite[Section 5]{bafrra:ecr}.

We are motivated by related calculations of the Borel equivariant 
cohomology of more general singular examples, in which 
$H_{T^n}^*(Y)$ is identified with the graded ring of piecewise 
polynomials on $\vsy$ \cite{bafrra:ecr}. In its simplest form, our 
main result states the following.
\begin{thm}\label{main1} 
For any divisive weighted projective space, $K_{T^n}^0(\Pc)$ is
isomorphic as a $K^0_{T^n}$-algebra to the ring of piecewise Laurent 
polynomials on $\vsy$; furthermore, $K_{T^n}^1(\Pc)$ is zero.
\end{thm}
A precise statement is proven as Theorem \ref{supmain}.1 in 
Section \ref{coap}.

More recently, Anderson and Payne \cite{anpa:okt} have introduced
equivariant \emph{operational} algebraic $K$-theory, and identified   
the rings op$K_{T^n}^\circ(Y)$ with piecewise Laurent
polynomials on $\vsy$. Their calculations are valid for \emph{all}
toric varieties, and therefore agree with ours on divisive 
weighted projective spaces.

Turning to $MU_{T^n}^*(\Pc)$, we note first that algebraic 
geometers have developed a successful theory of \emph{algebraic 
cobordism} during the last 15 years, by working over the Lazard 
ring $L^*$. They have also introduced equivariant versions that are 
related to $MU_{T^n}^*(-)$. As described by Krishna and Uma 
\cite{krum:acr}, for example, these theories are complete; so their 
coefficient rings cannot be isomorphic to $MU_{T^n}^*$. 
Nevertheless, the equivariant algebraic cobordism ring of 
many toric varieties $Y$ may be expressed in terms of 
piecewise formal power series on $\vsy$ \cite{krum:acr}. As 
we explain below, this is isomorphic to the Borel equivariant cobordism 
ring $MU^*(ET^n\times_{T^n}Y)$ in cases such as smooth $Y$, or 
products of weighted projective spaces.

Our conclusions for complex cobordism are based on the fact that 
$MU_{T^n}^*$ is an algebra over the Lazard ring $L^*$, graded 
cohomologically. So we refer to $MU_{T^n}^*$ as the ring of 
\emph{$T^n$-cobordism forms}, and express our second result 
accordingly.
\begin{thm}\label{main2} 
For any divisive weighted projective space, $MU_{T^n}^*(\Pc)$ is
isomorphic as an $MU^*_{T^n}$-algebra to the ring of piecewise
cobordism forms on $\vsy$; in particular, $MU_{T^n}^*(\Pc)$ is zero in
odd dimensions.
\end{thm}
A precise statement is proven as Theorem \ref{supmain}.2 in 
Section \ref{coap}.

Most recently, inspired by \cite{anpa:okt}, Gonzalez and 
Karu \cite{goka:bac} have defined equivariant \emph{operational}
algebraic cobordism. For any quasiprojective toric variety $Y$, 
they identify their operational ring with the ring of piecewise 
formal power series on $\vsy$, and therefore with 
$MU^*(ET^n\times_{T^n}Y)$ for smooth $Y$, or products of weighted 
projective spaces.

We introduce weighted projective spaces as singular toric varieties in
Section \ref{weprsp}, focusing on divisive examples $\Pc$ and their
invariant CW-structures. In Section \ref{gegkmth} we recall the
generalised GKM-theory that allows us to compute $K_{T^n}^*(Y)$ 
and $MU_{T^n}^*(Y)$
for certain stratified $T^n$-spaces $Y$, and confirm that the theory
applies to $\Pc$. In order to rewrite the outcome in the language of
Theorems \ref{main1} and \ref{main2},
we devote Section \ref{pial} to describing diagrams of algebras, 
and piecewise structures on arbitrary fans. We combine
the two viewpoints in Section \ref{coap}, and deduce a version of
Theorems \ref{main1} and \ref{main2}
that holds for a wider class of equivariant
cohomology theories. In Section \ref{coboco} we relate
$K_{T^n}^*(\Pc)$ and $MU_{T^n}^*(\Pc)$ to the Borel equivariant 
$K$-theory and cobordism 
of $\Pc$, in terms of piecewise formal power series, the Chern
character, and the Boardman homomorphism.
Finally, in Section \ref{smca}, we extend our conclusions to 
\emph{smooth} toric varieties, and rewrite the resulting piecewise 
algebras in the context of face rings.

Before we begin, it is convenient to introduce notation and conventions 
that we shall use without further comment.

We write $S^1$ for the circle as a topological space, and
$T<\C^\times$ for its realisation as the group of unimodular complex
numbers under multiplication. The $(n+1)$--dimensional compact torus
$T^{n+1}$ is a subgroup of the algebraic torus $(\C^\times)^{n+1}$,
and acts on $\C^{n+1}$ by coordinate-wise multiplication. This is the
\emph{standard action}; it preserves the unit sphere
$S^{2n+1}\subset\C^{n+1}$, and the corresponding orbit space may be
identified with the standard $n$-simplex $\varDelta^n\subset\brpnp$ in
the positive orthant.

Readers who require background information and further references on
equivariant topology may consult \cite{alpu:cmt}, and the survey
articles in \cite{may:ehc}. For fans, toric varieties, and their
topological aspects, we suggest \cite{fra:dtv}, \cite{ful:itv} and
\cite{oda:cba}.

{\bf Acknowledgments.} We are especially grateful to Dave Anderson and Sam Payne for their
advice and encouragement on matters of algebraic geometry, to Tony
Bahri and Matthias Franz for many hours of discussion on the topology
of weighted projective spaces, and to the Hausdorff Institute in Bonn
for supporting our work together on this and related projects. Megumi
Harada was partially supported by an Early Researcher Award from the 
Ministry of Research and Innovation of Ontario, and a Discovery Grant 
from the NSERC of Canada; Tara Holm was partially supported by a PCCW 
Affinito-Stewart Grant, Grant 208975 from the Simons Foundation, and 
NSF Grant DMS--1206466.

%
%
%
%
%
%
%
%
%

\section{Weighted projective space}\label{weprsp}

Weighted projective spaces are amongst the simplest and most elegant
examples of \emph{toric orbifolds} \cite{ful:itv}, and we devote this
section to summarising their definition and basic properties.

A {\em weight vector} $\chi$ is a sequence $(\chi_0,\dots,\chi_n)$ of 
$n+1$ positive integers; $\chi$ determines a subcircle $\Tc< T^{n+1}$ 
by 
\begin{equation*} 
\Tc\;=\;\{(t^{\chi_0},\dots,t^{\chi_n}):|t|=1\}, 
\end{equation*} 
which acts on $S^{2n+1}$ with finite stabilizers. Then the weighted
projective space $\Pc$ is defined to be the orbit space
$S^{2n+1}/T(\chi)$. Each point of $\Pc$ may be written as an
equivalence class $[z]=[z_0,\dots,z_n]$, where the $z_j$ are known as
{\em homogenous coordinates}. Permutations of the $z_j$ induce
self-homeomorphisms of $\Pc$, so we may reorder the weights as
required; it is often convenient to assume that they are
non-decreasing. Of course $\Pc$ may equally well be exhibited as the
orbit space $(\C^{n+1}\setminus\{0\})/\C^\times(\chi)$, and therefore
as a complex algebraic variety.

The finite stabilisers ensure that $\Pc$ is an {\em orbifold}, which
is singular for $n>1$ unless $\chi=(d,\dots,d)$ for some positive
integer $d$.  The residual action of the torus $T^n\cong T^{n+1}/\Tc$
turns $\Pc$ into a {\em toric} orbifold, with quotient polytope an
$n$-simplex. If $\chi=(1,\dots,1)$, then $\Tc$ is the diagonal circle
$T_\delta<T^{n+1}$, and $\Pc$ reduces to the standard projective space
$\CPn$. In this case, $K_{T^n}^*(\CPn)$ is computed in 
\cite[\S 4.1]{grwi:pdk}.

Orbifolds may be studied from a number of different perspectives, and
recent articles have focused on their interpretation as {\em
  groupoids} \cite{moe:ogi} and as {\em stacks}
\cite{ler:oas}. Several invariants of these richer structures have
been defined, such as the orbifold fundamental group
\cite{adleru:ost}.  Nevertheless, in this work we remain firmly in the
topological world, and study the underlying topological space $\Pc$;
in other contexts, it is known as the {\em coarse moduli space} of the
stack. There has been a recent surge of interest \cite{bafrra:ecr},
\cite{bafrnora:cwp} in its $T^n$-equivariant topological invariants.

By construction, $\bP(d\chi)$ is equivariantly homeomorphic to
$\bP(\chi)$ for any positive integer $d$, although they differ as
orbifolds. For our purposes, it therefore suffices to assume that the
greatest common divisor $\gcd(\chi)$ is $1$; in orbifold terminology,
this is tantamount to restricting attention to \emph{effective}
cases. If $\gcd(\chi)=1$, then \cite{dol:wpv} provides an equivariant
homeomorphism
\[
\bP(d\chi_0,\dots,d\chi_{j-1},\chi_j,d\chi_{j+1},\dots,d\chi_n)
\;\cong\;\Pc
\]
for any $0\leq j\leq n$, and any positive integer $d$ such that
$\gcd(d,\chi_j)=1$. Further simplification is therefore possible by 
insisting that $\chi$ be {\em normalised}, in the sense that
\begin{equation}\label{norm}
\gcd(\chi_0,\dots,\widehat{\chi_j},\dots,\chi_n)\;=\;1
\end{equation}
for every $0\leq j\leq n$. 

We may impose additional restrictions on the weights, with important 
topological consequences.
\begin{defns}\label{fidivtodiv}
The weight vector $\chi$ and the weighted projective space $\Pc$ are
\begin{enumerate}
\item[{\bf 1.}] 
\emph{weakly divisive} if $\chi_j$ divides $\chi_n$ for every $0\leq j<n$,
\item[{\bf 2.}] 
\emph{divisive} if $\chi_{j-1}$ divides $\chi_j$ for every 
$1\leq j\leq n$.
\end{enumerate}
\end{defns}
A divisive $\chi$ is automatically weakly divisive, and is necessarily
non-decreasing. Moreover, $\chi$ is divisive precisely when the
reverse sequence $\chi_n$, \dots, $\chi_0$ is {\em well ordered}, in
the sense of Nishimura and Yosimura \cite{niyo:qkt}. 
\begin{rem}[{\cite[Theorem 3.7, Corollary 3.8]{bafrra:wps}}]
\label{fact:weak-thom} If $\Pc$ is weakly divisive, then it is 
homeomorphic to the Thom space of a complex line bundle over $\Pd$,
where $\chi' = (\chi_0,\dots,\chi_{n-1})$; if it is divisive, then it
is homeomorphic to an $n$-fold iterated Thom space of complex line
bundles over the one-point space $*$. 
\end{rem}

In case $\chi_0=1$, there exists a canonical isomorphism $j\colon T^n\to
T^{n+1}/\Tc$, defined by setting
$j(u_1,\dots,u_n)=[1,u_1,\dots,u_n]$; the resulting action of $T^n$ on
$\Pc$ satisfies
\begin{equation}\label{canact}
(u_1,\dots,u_n)\cdot[z_0,z_1,\dots,z_n]\;=\;[z_0,u_1z_1,\dots,u_nz_n].
\end{equation}

From this point onwards, we therefore make the following assumptions.
\begin{assumptions}\label{assdivchicanact}\hfill{}
\begin{enumerate}
\item[{\bf 1.}]
The weight vector $\chi$ is both normalised and divisive, so
$\chi_0=\chi_1=1$.
\item[{\bf 2.}]
The residual action $T^n\acts\Pc$ is given by the isomorphism $j$ and 
\eqref{canact}.
\end{enumerate}
\end{assumptions}

For {\em any} weighted projective space, Kawasaki's calculations
\cite{kaw:ctp} imply the existence of a homotopy equivalent CW-complex
with a single cell in every even dimension. This has been realized in
\cite{bafrnora:cwp}, but current evidence suggests that an explicit
cellular decomposition for the general case is unpleasantly
complicated \cite{oni:phd}. Nevertheless, Remark \ref{fact:weak-thom}
provides an easy solution for divisive $\chi$. Given any $1\leq i\leq
n$, it is convenient to write $D^{2i}$ for the closed unit disc
\[
\{w:|w_{n-i+1}|^2+\dots+|w_n|^2\leq 1\}\;\subset\;\C^i,
\]
and $g_i\colon D^{2i}\ra\Pc$ for the map given by 
\begin{equation}\label{attachmap}
g_i(w)\;=\;
\big[0,\dots,0,(1-|w_{n-i+1}|^2-\dots-|w_n|^2)^{1/2},w_{n-i+1},\dots,
w_n\big].
\end{equation}
For $i=0$, let $D^0=\{0\}$ and $g_0(0)=[0,\dots,0,1]$.

\begin{prop}\label{CWstruc}
For every divisive $\Pc$, the $g_i$ are characteristic maps for a 
CW-structure that contains exactly $n+1$ cells. 
\end{prop}
\begin{proof}
For $1\leq i\leq n$, the restriction of $g_i$ to the interior of 
$D^{2i}$ is a homeomorphism onto
\begin{equation}\label{ciconds}
\{[z]:z_0=\dots=z_{n-i-1}=0,\,z_{n-i}\neq 0\}\;\subset\;\Pc,
\end{equation}
which is therefore an open $2i$-cell.
Furthermore, $g_i$ maps the boundary of $D^{2i}$ onto the subspace 
$\{[z]:z_0=\dots=z_{n-i}=0\}$, which is the union of all lower 
dimensional cells. The zero cell is $\{[0,\dots,0,1]\}$. 
\end{proof}
\begin{cor}\label{invariantcells}
The CW-structure is invariant under the residual action of $T^n$.
\end{cor}
\begin{proof}
The action \eqref{canact} automatically preserves the conditions of 
\eqref{ciconds}.
\end{proof}

Combining \eqref{canact} and \eqref{attachmap} shows that the 
characteristic map $g_i$ induces the action 
\begin{equation}\label{repsci}
(u_1,\dots,u_n)\cdot(w_{n-i+1},\dots,w_n)\;=\;
(u_{n-i+1}u_{n-i}^{-\chi_{n-i+1}/\chi_{n-i}}w_{n-i+1},\dots,
u_nu_{n-i}^{-\chi_n/\chi_{n-i}}w_n)
\end{equation}
of $T^n$ on $D^{2i}$, for each $1\leq i\leq n$ (taking $u_0=1$ in case
$i=n$). This is the unit disc $D(\rho_i)$ of an $i$-dimensional
unitary representation $\rho_i$ of $T^n$.

We denote the CW-structure of Proposition \ref{CWstruc} by
$\Pc=e^0\cup e^2\cup\dots\cup e^{2n}$, where $e^{2i}$ is the closure
of \eqref{ciconds} in $\Pc$; the centers $[0,\dots,0,1,0,\dots,0]$ of
the cells are precisely the fixed points of the residual action.

%
%
%
%
%
%
%
%
%

\section{Generalized GKM-theory}\label{gegkmth}

In this section we recall the generalized GKM-theory of \cite{HHH},
and explain its application to Corollary \ref{invariantcells}. This
leads to a description of $E_{T^n}^*(\Pc)$ for divisive $\chi$ and
several examples of $E_{T^n}$, including equivariant 
complex $K$-theory and cobordism.

Following \cite[\S3]{HHH}, we require the space $G\acts Y$ to be
equipped with a $G$-invariant stratification $Y=\bigcup_{i\in I} Y_i$,
and write $Y_{<i}$ for the subspace ${\bigcup_{j<i}Y_j\subset Y_i}$
for every $i\in I$. We insist that each $Y_i$ contains a subspace
$F_i$, whose neighbourhood is homeomorphic to the total space $V_i$ of
a $G$-equivariant $E_G$-oriented vector bundle $\rho_i\letbe
(V_i,\pi_i,F_i)$. As usual, the {\em equivariant Euler class} 
$e_G(\rho_i)$ is defined in \smash{$E_G^{\dim V_i}(F_i)$} by 
restricting the chosen equivariant Thom class of $\rho_i$ to the 
zero section, for every $i\in I$.

We recall the following four assumptions of \cite{HHH}, which insure 
that $E_G^*(Y)$ may be computed by their methods. As we shall see, 
they are satisfied by every divisive $\Pc$.
\begin{assumptions}\label{asss}\hfill{}
\begin{enumerate}
\item[{\bf 1.}]
Each subquotient $Y_i/Y_{<i}$ is homeomorphic to the Thom space 
$\Th(\rho_i)$, with attaching maps $\varphi_i:S(\rho_i)\to Y_{<i}$.
\item[{\bf 2.}]
Every $\rho_i$ admits a $G$-equivariant direct sum decomposition
$\bigoplus_{j<i}\rho_{i,j}$ into $E_G$-oriented subbundles 
$\rho_{i,j}=(V_{i,j},\pi_{i,j},F_i)$.
\item[{\bf 3.}]
There exist $G$-equivariant maps $f_{i,j}\colon F_i\to F_j$ such that 
the restrictions $f_{i,j}\scirc\pi_{i,j}|_{S(\rho_{i,j})}$ and 
$\varphi_i|_{S(\rho_{i,j})}$ agree for every $j<i$.
\item[{\bf 4.}]
The Euler classes $e_G(\rho_{i,j})$ are not divisors of zero in 
$E^*_G(F_i)$ for any $j<i$, and are pairwise relatively prime. 
\end{enumerate}
\end{assumptions}
Note that the $\rho_{i,j}$ may have dimension $0$. Assumption
\ref{asss}.4 means that $e_G(\rho_{i,j})$ divides $y$ for each $j$ if
and only if $e_G(\rho_i)$ divides $y$, for any $y\in E^*_G(F_i)$.

Now let $\iota^*\colon E_G^*(Y)\to\prod_i E_G^*(F_i)$ be the 
homomorphism induced by the inclusion $\coprod_iF_i\subset Y$.
\begin{thm}[{\cite[Theorem 3.1]{HHH}}] \label{th:GKM} 
Let $Y$ be a $G$-space satisfying the four Assumptions~\ref{asss}; then 
$\iota^*$ is monic, and has image
\begin{equation*}
\varGamma_{\,Y}\:\mbox{\em $\letbe$}\;\left\{y=(y_i):
e_G(\text{$\rho_{i,j})$ divides $y_i-f_{i,j}^*(y_j)$ 
for all $j<i$}\right\}\;\leq\;\prod_i E^*_G(F_i).
\qedhere\popQED
\end{equation*}
\end{thm}

As in several of the examples in \cite{HHH}, our application to
Corollary \ref{invariantcells} involves a $T^n$-invariant skeletal
filtration. Specifically, $Y_i=\cup_{j\leq i}e^{2j}$ is the
$2i$--skeleton of $\Pc$ for $0\leq i\leq n$, and the $F_i\subset Y_i$
contain only the centers of the cells $e^{2i}$. These are the
fixed points of the $T^n$-action, and the $T^n$-equivariant bundles
$\rho_i$ reduce to $i$--dimensional complex representations, which are
canonically $E_{T^n}$-oriented. Assumption \ref{asss}.1 is then
satisfied, where the r\^oles of the $\varphi_i$ are played by the 
restrictions of the $g_i$ of Proposition \ref{CWstruc} to $S^{2i-1}$. 
The equivariant Euler classes $e_{T^n}(\rho_i)$ lie in the coefficient 
ring $E_{T^n}^*$.

In order to check Assumption \ref{asss}.2, we refer back to
\eqref{repsci}. Each $\rho_i$ decomposes as a sum
$\bigoplus_{j<i}\rho_{i,j}$ of $1$--dimensionals, where $\rho_{i,j}$
is defined by
\begin{equation}\label{repsdij}
(u_1,\dots,u_n)\cdot w_{n-j}\;=\;
u_{n-j}u_{n-i}^{-\chi_{n-j}/\chi_{n-i}}w_{n-j}
\end{equation}
for $0\leq j< i\leq n$. These decompositions respect the canonical 
$E_{T^n}$-orientations, by definition. 

For Assumption \ref{asss}.3, the maps $f_{i,j}$ are necessarily constant 
and equivariant, so the restrictions to $S(\rho_{i,j})$ of 
$f_{i,j}\scirc\pi_{i,j}$ and $g_i$ agree, for every $j<i$.

Before confirming Assumption~\ref{asss}.4, recall \cite{hus:fb} that the 
complex representation ring of $T^n$ is isomorphic to the Laurent 
polynomial algebra
\begin{equation}\label{deflpa}
R(T^n)\;\cong\;S_\Z^\pm(\alpha)
\;\letbe\;\Z[\alpha_1,\dots,\alpha_n]_{(\alpha_1\cdots\alpha_n)}
\end{equation}
on generators $\alpha_j$, which represent the $1$--dimensional 
irreducibles defined by projection onto the $j$th coordinate circle.
In particular, \eqref{repsdij} states that 
\begin{equation}\label{dijalphas}
\rho_{i,j}\;\cong\;\alpha_{n-j}\alpha_{n-i}^{-\chi_{n-j}/\chi_{n-i}}
\end{equation}
(taking $\alpha_0=1$ in case $i=n$).
Since equivariant Euler classes behave exponentially, $e_{T^n}$ is 
determined on any representation by its value on the monomials 
\smash{$\alpha^J\letbe\alpha_1^{j_1}\cdots\alpha_n^{j_n}$}, which 
form an additive basis for $R(T^n)$ as $J$ ranges over $\Z^n$.

In the case of $K$-theory, the coefficient ring $K_{T^n}^*$ is
isomorphic to $R(T^n)$ in even dimensions, and is zero in odd
\cite{Seg68}. Periodicity may be made explicit by incorporating the
Bott element $z$ into \eqref{deflpa} and writing 
\begin{equation}\label{defktn}
K_{T^n}^*\;\cong\;S_{K^*}^\pm(\alpha)\;\cong\;R(T^n)[z,z^{-1}]\,, 
\end{equation}
where $\alpha^J$ (for any $J$) and $z$ have cohomological dimensions
$0$ and $-2$ respectively. The notation reflects the fact that
the coefficient ring $K^*$ is isomorphic to $\Z[z,z^{-1}]$. Then 
$e_{T^n}(\alpha^J)=1-\alpha^{\pm J}$, where both choices of sign occur 
in the literature. Some authors even prefer $z^{-1}(1-\alpha^{\pm J})$, 
to achieve greater consistency with cobordism and cohomology by 
realizing the Euler class in cohomological dimension 2; here, 
for notational convenience, we employ $1-\alpha^J$. The kernel 
of the augmentation $K^*_{T^n}\to K^*$ is the ideal 
$(1-\alpha_1,\dots,1-\alpha_n)$.

In the case of complex cobordism, the coefficient ring $MU_{T^n}^*$ is
an algebra over $L^*$, and is freely generated as an $L^*$-module by
even-dimensional elements \cite{com:cce}. The Euler classes
$e_{T^n}(\alpha^J)$ are non-zero elements of $MU_{T^n}^2$, and are
denoted by $e(\alpha^J)$ in the calculations of \cite{sin:cce} and
elsewhere; they generate the positive--dimensional subring
$MU_{T^n}^{>0}$. The kernel of the augmentation $MU_{T^n}^*\to L^*$ is
the ideal $(e(\alpha_1),\dots,e(\alpha_n))$ \cite{coma:ctc}.

In the case of Borel equivariant integral cohomology, the 
coefficient ring $H_{T^n}^*$ is isomorphic to the polynomial algebra
\begin{equation}\label{defpolyalgx}
H^*(BT^n;\Z)\;\cong\;S_\Z(x)\;\letbe\;\Z[x_1,\dots,x_n]
\end{equation}
on $2$--dimensional generators $x_j$. Then 
$e_{T^n}(\alpha^J)=\sum_Jj_kx_k$ for any $J$; in particular, the 
equation $e_{T^n}(\alpha_i)=x_i$ may be taken to define $x_i$ for 
every $1\leq i\leq n$. The kernel of of the augmentation 
$H^*_{T^n}\to H^*$ is the ideal $(x_1,\dots,x_n)$.

So from \eqref{dijalphas}, we deduce that 
\begin{equation}\label{etdij}
e_{T^n}(\rho_{i,j})\;=\;\left\{\begin{array}{ll}
\smash{1-\alpha_{n-j}\alpha_{n-i}^{-\chi_{n-j}/\chi_{n-i}}\;\;
\text{in $K_{T^n}^0$}}\\
\rule{0em}{1.3em}e(\alpha_{n-j}\alpha_{n-i}^{-\chi_{n-j}/\chi_{n-i}})\;\;
\text{in $MU_{T^n}^2$}\\
\rule{0em}{1.3em}x_{n-j}-(\chi_{n-j}/\chi_{n-i})x_{n-i}\;\;
\text{in $H_{T^n}^2$}
\end{array}\right.
\end{equation}
for $0\leq j<i\leq n$. 

In each of these three cases, the ambient ring is an integral domain;
for $MU_{T^n}^*$, this is proven in \cite[Corollary 5.3]{sin:cce}. So
none of the Euler classes of \eqref{etdij} are divisors of zero. The
following criteria address the remaining parts of Assumption
\ref{asss}.4.
\begin{cria}[{\cite[Lemma 5.2]{HHH}}]\label{cricoll} 
Given any finite set of non-zero $\alpha^J$, their equivariant Euler 
classes are pairwise relatively prime in $K_{T^n}^*$ or $MU_{T^n}^*$
whenever no two $J$ are linearly dependent in $\Z^n$;
the additional condition that no prime $p$ divides any two $J$ is
required in $H_{T^n}^*$. 
\end{cria}

For $\rho_{i,j}$ with $i<n$, \eqref{dijalphas} shows that $J$ has only
two non-zero entries, namely $1$ in position $n-j$ and
$-\chi_{n-j}/\chi_{n-i}$ in position $n-i$; for $\rho_{n,j}$, there is
a single $1$ in position $n-j$. So Criteria \ref{cricoll} confirm that
the Euler classes $e_{T^n}(\rho_{i,j})$ are pairwise relatively prime in
all three cases, and therefore that Assumption~\ref{asss}.4 also
holds.

We may now conclude our first description of $E_{T^n}^*(\Pc)$.
\begin{prop}\label{gkmdivwps}
For any divisive weighted projective space, $E_{T^n}^*(\Pc)$ is
isomorphic as an $E^*_{T^n}$-algebra to the subring 
\[
\varGamma_{\,\bP(\chi)}\;=\;\left\{y:e_{T^n}(\text{$\rho_{i,j})$ divides
  $y_i-y_j$ for all $j<i$} \right\}\;\leq\;\prod_i
E^*_{T^n},
\]
in each of the cases $E=K$, $\MU$, $H$.
\end{prop}
\begin{proof}
Our preceding analysis shows that Theorem \ref{th:GKM} applies
directly to the skeletal filtration. Compatibility with the 
$E^*_{T^n}$-algebra structure follows immediately. 
\end{proof}
Proposition \ref{gkmdivwps} shows that $E_{T^n}^*(\Pc)$ is zero in odd 
dimensions.

The idea behind Theorem \ref{supmain} is to convert Proposition 
\ref{gkmdivwps} into a form more directly related to the properties 
of the fan $\vsc\letbe\varSigma_{\bP(\chi)}$.

%
%
%
%
%
%
%
%
%

\section{Piecewise algebra}\label{pial}

Before stating Theorem \ref{supmain}, we introduce certain algebraic
and geometric constructions that are associated to fans by the theory
of diagrams. They are motivated by modern approaches to homotopy
theory, and provide a common language in which to address the cases
under discussion.

A rational fan $\varSigma$ in $\R^n$ determines a small category
$\cats$, whose objects are the cones $\sigma$ and morphisms their
inclusions $i_{\sigma,\tau}\colon\sigma\subseteq\tau$. The zero cone
$\{0\}$ is initial, and the maximal cones admit only identity
morphisms. The opposite category $\catsop$ has morphisms
$p_{\tau,\sigma}\colon\tau\supseteq\sigma$, and $\{0\}$ is final.

For $0\leq d\leq n$, the set of $d$--dimensional cones is denoted by
$\varSigma(d)\subseteq\varSigma$. The elements of $\varSigma(1)$ are
known as \emph{rays}, and are represented by primitive vectors 
$v_1$, \dots, $v_m$, where $m$ denotes the cardinality of 
$\varSigma(1)$ henceforth. Each cone may be identified by its
generating rays $v_{j_1}$, \dots, $v_{j_k}$, and interpreted as a
subset $\sigma\subseteq\varSigma(1)$. The cardinality $k=|\sigma|$
coincides with the dimension $d=\dim(\sigma)$ if and only if the cone 
$\sigma$ is simplicial.

Every $d$-dimensional $\sigma$ gives rise to an $(n-d)$--dimensional
subspace $\R_{\sigma^\perp}\!\subseteq\R^n$, by forming the orthogonal
complement of its linear hull $\R_\sigma$. The rationality of $\sigma$ 
implies that $\R_{\sigma^\perp}\1\cap\Z^n$ has rank $(n-d)$, and admits a 
basis $w_1$, \dots, $w_{n-d}$ of integral vectors; it is unique up to 
the action of $\GL(n-d,\Z)$, and determines the linear forms
\begin{equation}\label{linforms}
w_c^{tr}x\;=\;w_{c,1}x_1+\cdots+w_{c,n}x_n\quad
\text{for $\;\,1\leq c\leq n-d$}.
\end{equation}
The intersection of their kernels is $\R_\sigma$, and there exists a 
splitting $\R^n\cong\R_{\sigma^\perp}\times\R_\sigma$. It is convenient to 
interpret $\R^n$ as the Lie algebra of $T^n$ and write the associated 
splitting as 
\begin{equation}\label{tnsplit}
T^n\;\cong\;T_{\sigma^\perp}\times T_\sigma,
\end{equation}
where the Lie algebra of $T_{\sigma^\perp}$ is spanned by the $w_c$,
for any cone $\sigma$. Thus $T_{\sigma^\perp}=\{1\}$ for 
top-dimensional cones, and $T_{\{0\}^\perp}=T^n$. 

\begin{defns}\label{defpd}
  A \emph{$\varSigma$-diagram} in a category $\cat{c}$ is a covariant
  functor $F\colon\cats\ra\cat{c}$; similarly, a 
  \emph{$\varSigma^{op}$-diagram} (or \emph{contravariant
  $\varSigma$-diagram}) is a covariant functor 
  $F\colon\catsop\ra\cat{c}$.
\end{defns} 
We are interested in diagrams for which one or both of $\lim F$ 
and $\col F$ exist in $\cat{c}$.

Definitions \ref{defpd} are motivated by a familiar diagram in
$\cat{top}$, which underlies the construction of the toric variety
$\xvs$ as a topological $T^n$-space. It is denoted by
$U\colon\cats\ra T^n$-\cat{top}, and uses the dual cones 
$\sigma^\vee$ in the lattice $M\letbe(\Z^n)^\vee$; it is given by
\begin{equation}\label{defxvsu}
U(\sigma)\;=\;U_\sigma\;\letbe\;
\Hom(\sigma^\vee\2\cap M\om,\om\C^\times\2\cup\{0\})
\sands U(i_{\sigma,\tau})\;=\;j_{\sigma,\tau}\,, 
\end{equation}
where $\Hom(-)$ denotes the affine variety of semigroup
homomorphisms, and $j_{\sigma,\tau}\colon U_\sigma\ra U_\tau$ 
is induced by $i^\vee_{\sigma,\tau}\colon\tau^\vee\ra\sigma^\vee$. 
The standard description of $\xvs$, as given by 
\cite[\S 1.4]{ful:itv}, for example, may then be expressed as 
the colimit $\col U$ in $T^n$-\cat{top}.

In fact $T^n\acts U_\sigma$ is $T^n$-equivariantly homotopy 
equivalent to $T^n\acts T^n/T_\sigma$ for every cone $\sigma$ 
\cite[Proposition 12.1.9, Lemma 3.2.5]{colisc:tv}. So $U$ is 
objectwise equivariantly equivalent to the diagram 
$V\colon\cats\ra T^n$-\cat{top}, given by 
\begin{equation}\label{defdiagv}
V(\sigma)\;=\;T^n/T_\sigma\sands 
V(i_{\sigma,\tau})\;=\;r_{\sigma,\tau}\,, 
\end{equation}
where $r_{\sigma,\tau}$ is the projection induced by the inclusion
$T_\sigma\leq T_\tau$. Since $U$ is cofibrant, it follows that 
$\hocolim V$ is equivariantly homotopy equivalent to 
$\col U=\xvs$. Diagram \eqref{defdiagv} first appeared in 
\cite{wezizi:hcc}, and more recently in \cite{fra:dtv}.

We may now describe our basic examples of 
$\varSigma^{op}$-diagrams in the category $\catgca_E$ of graded 
commutative $E^*_{T^n}$-algebras.
\begin{defn}\label{etdiag}
For any complex-oriented equivariant cohomology theory 
$E^*_{T^n}(-)$, the diagram 
$\EV\colon\catsop\to\catgca_E$ has
\begin{equation}\label{polysig}
\EV(\sigma)\;=\;E^*_{T^n}(T^n/T_\sigma)\sands
\EV(p_{\tau,\sigma})\;=\;r^*_{\sigma,\tau}\,.
\end{equation}
The limit $P_E(\varSigma)$ of $\EV$ is the $E^*_{T^n}$-algebra 
of \emph{piecewise coefficients} on $\varSigma$. 
\end{defn}
\begin{rems}
By definition, $P_E(\varSigma)$ is an $E^*_{T^n}$-subalgebra of
$\prod_\sigma\EV(\sigma)$, so every piecewise coefficient $f$ has
one component $f^\sigma$ for each cone $\sigma$ of $\varSigma$. 
If $\sigma$ is top dimensional, then $T_\sigma=T^n$ and $f^\sigma$ is 
a genuine element of $E^*_{T^n}$; on the other hand, 
$T_{\{0\}}=\{1\}$ and $f^{\{0\}}$ lies in $E^*$. The components of 
$f$ are  compatible under the homomorphisms $j^*_{\tau,\sigma}$, and 
are congruent modulo the augmentation ideal. Sums and products of 
piecewise coefficients are taken conewise, and 
$E_{T^n}^*\leq P_E(\varSigma)$ occurs as the subalgebra of 
\emph{global coefficients}, whose components agree on every cone. 
In particular, it contains the global constants $0$ and $1$, which 
act as zero and unit respectively.
\end{rems}

In many cases, $\EV$ and $P_E$ may be described more explicitly, as 
follows.

Suppose that $\rho$ has codimension $1$, and that $w_1$ 
is a primitive vector generating $\R_{\rho^\perp}$. The splitting 
\eqref{tnsplit} ensures that the natural action of $T^n$ on 
$T^n/T_\rho$ may then be identified with the unit circle 
$S(\eta)$ of the irreducible representation $\eta\letbe\alpha^{w_1}$, 
on which the circle $T_{\rho^\perp}$ acts freely and the 
$(n-1)$--torus $T_\rho$ acts trivially. The inclusion of $S(\eta)$ into 
the unit disc $D(\eta)$ determines the equivariant cofiber sequence
\[
S(\eta)\;\lra\; D(\eta)\;\lra\; S^{\eta},
\] 
where $S^{\eta}$ denotes the one-point compactification 
$T^n\acts D(\eta)/S(\eta)$. Applying $E^*_{T^n}(-)$ yields the long 
exact sequence
\begin{equation}\label{lessd}
\cdots\;\lra\; E^*_{T^n}(S^{\eta})\;\stackrel{\cdot e}{\lra}\; 
E^*_{T^n}(D(\eta))
\;\lra\; E^*_{T^n}(S(\eta))\;\lra\;\cdots\;.
\end{equation}
Since $D(\eta)$ is equivariantly contractible and the Thom 
isomorphism applies to the Thom space $S^{\eta}$, the 
homomorphism $\cdot e$ may be interpreted as multiplication by the 
Euler class $e_{T^n}(\eta)$. So $\cdot e$ is monic for each of $E=K$,
$\MU$, $H$; thus \eqref{lessd} becomes short exact, yielding 
isomorphisms
\[
E^*_{T^n}/(e_{T^n}(\eta))\;\cong\;
E^*_{T^n}(S(\eta))\;\cong\; E^*_{T^n}(T^n/T_\rho)\;=\;\EV(\rho)
\]
of $E_{T^n}^*$-algebras.

This calculation extends to lower dimensional cones $\tau$ by 
iteration. If $\tau$ has dimension $k$, then the natural action of $T^n$ 
on $T^n/T_\tau$ may be identified with the product 
$S(\eta_1)\times\dots\times S(\eta_{n-k})$, where 
$\eta_c$ denotes the irreducible $\alpha^{w_c}$ for $1\leq c\leq n-k$. 
The $(n-k)$--torus $T_{\tau^\perp}$ acts freely, and the $k$--torus 
$T_\tau$ acts trivially, yielding isomorphisms
\begin{equation}\label{eisod}
E^*_{T^n}/(e_{T^n}(\eta_1),\dots,e_{T^n}(\eta_{n-k}))
\;\cong\; E^*_{T^n}(S(\eta_1)\times\dots\times S(\eta_{n-k}))
\;\cong\; E^*_{T^n}(T^n/T_\tau)\;=\;\EV(\tau). 
\end{equation}
If $\sigma\subset\tau$ has dimension $d<k$, then 
$\R_{\sigma^\perp}$ arises from $\R_{\tau^\perp}$ by adjoining additional 
basis vectors $w_{n-k+1}$, \dots, $w_{n-d}$, and the projection
\[
q_{\tau,\sigma}\colon 
E^*_{T^n}/(e_{T^n}(\eta_1),\dots,e_{T^n}(\eta_{n-k}))
\;\lra\; E^*_{T^n}/(e_{T^n}(\eta_1),\dots,e_{T^n}(\eta_{n-d}))
\]
corresponds to $r^*_{\tau,\sigma}\colon 
E^*_{T^n}(T^n/T_\tau)\ra E^*_{T^n}(T^n/T_\sigma)$
under \eqref{eisod}.

We conclude that \eqref{polysig} may be rewritten as
\begin{equation}\label{qolysig}
\EV(\sigma)\;=\;E^*_{T^n}/(e_{T^n}(\eta_1),\dots,e_{T^n}(\eta_{n-d}))
\sands
\EV(p_{\tau,\sigma})\;=\; q_{\tau,\sigma},
\end{equation}
and proceed to describing the examples $E=K$, $\MU$, $H$ in these
terms.

For $E=K$, we work with graded commutative algebras over the 
Laurent polynomial ring $S_{K^*}^\pm(\alpha)$ of \eqref{defktn}.
\begin{exa}\label{defplp}
The \emph{Laurent polynomial diagram}
$KV\colon\catsop\ra\catgca_K$ has
\begin{equation}\label{laurpolysig}
KV(\sigma)\;=\;S_{K^*}^\pm(\alpha)/J_\sigma\sands
KV(p_{\tau,\sigma})\;=\;q_{\tau,\sigma}\,,
\end{equation}
where $J_\sigma$ denotes the ideal generated by the Euler classes
$1-\alpha^{w_c}$ arising from the $w_c$ of \eqref{linforms} for $1\leq
c\leq n-d$. In this case, $P_K(\alpha;\varSigma)$ is the
$S_{K^*}^\pm(\alpha)$-algebra of \emph{piecewise Laurent polynomials}
on $\varSigma$.
\end{exa}

For $E=\MU$, we work with graded commutative algebras over 
$\MU_{T^n}^*$, whose structure is unknown. We therefore rely on
the fact that every element of $\MU_{T^n}^*$ is an even-dimensional 
linear combination of generators over $L^*$, and refer to 
$\MU_{T^n}^*$ as the ring of \emph{$T^n$-cobordism forms}. Such forms 
may not be representable by genuine $T^n$-manifolds, as exemplified 
by the Euler class $e(\alpha^J)$, whose homological dimension is 
$-2$. This phenomenon arises from the lack of equivariant 
transversality, and the consequent failure of the Pontryagin-Thom 
construction to be epimorphic.

\begin{exa}\label{defbf}
The \emph{cobordism form diagram} 
$\MU V\colon\catsop\ra\catgca_{MU}$ has
\begin{equation}\label{borforsig}
  \MU V(\sigma)\;=\;\MU_{T^n}^*/J_\sigma\sands
  \MU V(p_{\tau,\sigma})\;=\;q_{\tau,\sigma}\,,
\end{equation}
where $J_\sigma$ denotes the ideal generated by the Euler classes
$e(\alpha^{w_c})$ for $1\leq c\leq n-d$. In this case, 
$P_{M\1U}(\varSigma)$ is the $\MU_{T^n}^*$-algebra of 
\emph{piecewise cobordism forms} on $\varSigma$.
\end{exa}

For $E=H$, we work with graded commutative algebras over the 
polynomial algebra $S_\Z(x)$ of \eqref{defpolyalgx}.
\begin{exa}\label{defpp}
The \emph{polynomial diagram} $HV\colon\catsop\ra\catgca_H$ 
has
\begin{equation*}
HV(\sigma)\;=\;S_\Z(x)/J_\sigma\sands
HV(p_{\tau,\sigma})\;=\;q_{\tau,\sigma}\,,
\end{equation*}
where $J_\sigma$ denotes the ideal generated by the Euler classes 
$w_c^{tr}x$ of \eqref{linforms} for $1\leq c\leq n-d$. In this case, 
$P_H(x;\varSigma)$ is the $S_\Z(x)$-algebra of 
\emph{piecewise polynomials} on $\varSigma$.
\end{exa}
In \cite{bafrra:ecr}, $P_H(x;\varSigma)$ is referred to as 
$\PP_\Z(x;\varSigma)$.

In Section \ref{coboco} we invest $H^*_{T^n}(-)$ with various
commutative integer-graded rings of coefficients $R$, which are zero
in odd dimensions. The standard example $\Z$ is concentrated in
dimension $0$, but we also consider $K\Q^*\letbe \Q[z,z^{-1}]$, where
$z$ has cohomological dimension $-2$, and
\[
\HwMU_*\;=\;\HwMU^{-*}\;\letbe\; H_*(MU)
\;\cong\;S_\Z(b_j:j\geq 1),
\]
where $b_j$ has cohomological dimension $-2j$ for every $j$. The
corresponding spectrum is denoted by $E=\HR$, and the analogue of 
diagram \eqref{laurpolysig} by $\HR V$; we continue to abbreviate 
$H\Z$ to $H$ in the standard example.

The equivariant coefficient ring $\HR^*_{T^n}=H^*(BT^n;R)$ must be 
identified with the completed tensor product $H^*(BT^n)\hot R$.
When $R=\Z$, the outcome is $H^*(BT^n)$; but for $K\Q^*$ 
or $\HwMU^*$, the ring 
\begin{equation}\label{hrcoeffs}
H^*(BT^n)\hot R\;\cong\;R[[x]]
\end{equation}
is an algebra of formal power series. It follows that 
$\HR V(\sigma)\cong R[[x]]/J_\sigma$, and that 
$P_{H\!R}(x;\varSigma)$ is the $R[[x]]$-algebra of 
\emph{piecewise formal power series} on $\varSigma$.

We require two further $\varSigma^{op}$-diagrams, obtained
by applying Definition \eqref{etdiag} to the Borel equivariant 
cohomology theories $E^*(ET^n\times_{T^n}-)$. In these cases the 
coefficients $E^*(BT^n)$ are also rings of formal power series.

The first such example identifies $K^*(BT^n)$ with  
$K^*[[\gamma]]$, on $0$--dimensional indeterminates $\gamma_j$ 
for $1\leq j\leq n$.
\begin{exa}\label{defbktd}
The \emph{Borel $K$-theory diagram}
$\KBV\colon\catsop\ra\catgca_{K\!_B}$ has
\begin{equation}\label{bksig}
\KBV(\sigma)\;=\;K^*[[\gamma]]\om/J_\sigma,
\end{equation}
where $J_\sigma$ denotes the ideal generated by the Euler classes
$(1+\gamma)^{w_c}-1$ for $1\leq c\leq n-d$ and 
$\KBV(p_{\tau,\sigma})$ is the natural projection. The limit 
$P_{K\!_B}(\gamma;\varSigma)$ is the $K^*[[\gamma]]$-algebra 
of piecewise formal power series on $\varSigma$.
\end{exa}

The second example identifies $\MU^*(BT^n)$ with 
$L^*[[u]]$, on indeterminates $u_j$ of cohomological dimension 
$2$ for $1\leq j\leq n$.
\begin{exa}\label{defbcd}
The \emph{Borel cobordism diagram}
$\MUBV\colon\catsop\ra\catgca_{MU\!_B}$ has
\begin{equation}\label{bmusig}
  \MUBV(\sigma)\;=\;L^*[[u]]\om/J_\sigma,
\end{equation}
where $J_\sigma$ denotes the ideal generated by the Euler classes
$[w_{c,1}](u_1)+_U\dots +_U[w_{c,n}](u_n)$ (expressed in 
terms of the universal formal group law $U$ \cite{haz:fga} over 
$L^*$), and $\MUBV(p_{\tau,\sigma})$ is the natural 
projection. The limit $P_{M\1U\!_B}(u;\varSigma)$ is the 
$L^*[[u]]$-algebra of piecewise formal power series on 
$\varSigma$.
\end{exa}

In fact $P_{M\1U}(\varSigma)$ and $P_{M\1U\?_B}(u;\varSigma)$ are 
the universal piecewise coefficient and piecewise formal power series 
algebras on $\varSigma$ respectively, for complex-oriented 
$E^*_{T^n}(-)$ and $E^*(ET^n\times_{T^n}-)$. The cases 
$P_K(\alpha;\varSigma)$ and $P_{K_B}(\gamma;\varSigma)$ 
correspond to the \emph{multiplicative} formal group law, classified 
by the equivariant Todd genus. Similarly, $P_{H\!R}(x;\varSigma)$ 
corresponds to the \emph{additive} formal group law, classified by 
the Thom genus.

\begin{rem}\label{functorial}
A map of fans $\xi\colon\varSigma'\to\varSigma$ may be interpreted as
an $n\times n'$ integer matrix $\xi$, for which the image $\xi(\sigma')$ 
of any cone $\sigma'$ is contained in some cone $\sigma$. Let 
$\xi^\dagger(\sigma')$ be the minimal such $\sigma$. In each of the 
above cases, $\xi$ induces a natural transformation
$(\xi^\dagger,\xi^*)$ of diagrams, and therefore a homomorphism
$\xi^*$ of limits.

For example, in the case of $HV$, the homomorphism $\xi^*\colon
HV(\sigma)\to HV(\sigma')$ is given in terms of the coordinate
functions $x$ and $x'$ by the matrix $\xi^{tr}$; it is well-defined
because $w\in\xi^\dagger(\sigma')^\perp$ implies that
$\xi^{tr}w\in(\sigma')^\perp$.  In the case of $KV$, the
homomorphism $\xi^*\colon KV(\sigma)\to KV(\sigma')$ is induced by
$\xi^*(\alpha^J)=(\alpha')^{\xi^{tr}J}$, and is well-defined for similar
reasons.

The construction of each piecewise algebra is therefore functorial
(although care is required to check that $\dagger$ preserves
composition). In particular, isomorphic fans yield isomorphic
algebras.
\end{rem}

Piecewise algebraic structures are natural generalizations of their
global counterparts, and provide simple qualitative descriptions of
algebras that may well be difficult to express in quantitative
terms. For example, see \cite[Section 4]{bafrra:ecr}, where
$P_H(x;\varSigma_{(1,2,3,4)})$ is computed in terms of 
generators and relations. The simplest non-trivial divisive example 
is the following.
\begin{exa}\label{oot}
The fan $\Sigma_{(1,1,2)}$ in $\R^2$ has seven cones: $\{0\}$; the
three rays through $r_0=(-1,-2)$, $r_1=(1,0)$, and $r_2=(0,1)$; and
three $2$-dimensional cones generated by all pairs of rays.

So $P_H(x)$ is an $S_\Z(x)$-algebra, where $x=(x_1,x_2)$. 
Calculations confirm that $P_H(x)$ is generated as a ring by four 
piecewise polynomials, namely the global linear functions $x_1$ and 
$x_2$, together with 
\[
\begin{picture}(295,66)
\put(0,30){$p\;=$}
\put(40,2){
\begin{picture}(40,60)
\put(20,40){\vector(-1,-2){20}}
\put(6,0){\makebox(0,0)[bl]{${\scriptstyle r_0}$}}
\put(20,40){\vector(1,0){20}}
\put(44,36){\makebox(0,0)[tr]{${\scriptstyle r_1}$}}
\put(20,40){\vector(0,1){20}}
\put(16,60){\makebox(0,0)[tr]{${\scriptstyle r_2}$}}
\put(28,48){\makebox(0,0)[bl]{$0$}}
\put(8,32){\makebox(0,0)[br]{$2x_1$}}
\put(20,16){\makebox(0,0)[bl]{$x_2$}}
\end{picture}}
\put(120,30){\rm and}
\put(170,30){$q\;=$}
\put(245,2){
\begin{picture}(40,60)\put(20,40){\vector(-1,-2){20}}
\put(6,0){\makebox(0,0)[bl]{${\scriptstyle r_0}$}}
\put(20,40){\vector(1,0){20}}
\put(44,36){\makebox(0,0)[tr]{${\scriptstyle r_1}$}}
\put(20,40){\vector(0,1){20}}
\put(16,60){\makebox(0,0)[tr]{${\scriptstyle r_2}$}}
\put(28,48){\makebox(0,0)[bl]{$0$}}
\put(16,32){\makebox(0,0)[br]{$x_1(2x_1-x_2)$}}
\put(20,16){\makebox(0,0)[bl]{$0$}}
\end{picture}}
\end{picture}
\]
of degree $2$ and $4$ respectively. In fact $P_H(x)$ is 
isomorphic to $\Z[x_1,x_2,p,q]/I_1$, where $I_1$ is the ideal
\[
\big(\,p(p-x_2)-2q,\;\,q(p-2x_1),\;\,
q(q-x_1(2x_1-x_2))\,\big)\,.
\]
As an $S_\Z(x)$--module, $P_H(x)$ has basis $\{1,p,q\}$.

Similarly, $P_K(\alpha)$ is an $S_\Z^\pm(\alpha)$-algebra, where
$\alpha=(\alpha_1,\alpha_2)$. Calculations confirm that 
$P_K(\alpha)$ is generated as a ring by six piecewise Laurent 
polynomials, namely the global elements $\alpha_1^\pm$ and 
$\alpha_2^\pm$, together with 
\[
\begin{picture}(320,66)
\put(0,30){$\epsilon\;=$}
\put(40,2){
\begin{picture}(40,60)
\put(20,40){\vector(-1,-2){20}}
\put(6,0){\makebox(0,0)[bl]{${\scriptstyle r_0}$}}
\put(20,40){\vector(1,0){20}}
\put(46,36){\makebox(0,0)[tr]{${\scriptstyle r_1}$}}
\put(20,40){\vector(0,1){20}}
\put(16,60){\makebox(0,0)[tr]{${\scriptstyle r_2}$}}
\put(28,48){\makebox(0,0)[bl]{$0$}}
\put(14,32){\makebox(0,0)[br]{$1-\alpha_1^2$}}
\put(16,16){\makebox(0,0)[bl]{$1-\alpha_2$}}
\end{picture}}
\put(125,30){\rm and}
\put(175,30){$\zeta\;=$}
\put(272,2){
\begin{picture}(40,60)\put(20,40){\vector(-1,-2){20}}
\put(6,0){\makebox(0,0)[bl]{${\scriptstyle r_0}$}}
\put(20,40){\vector(1,0){20}}
\put(44,36){\makebox(0,0)[tr]{${\scriptstyle r_1}$}}
\put(20,40){\vector(0,1){20}}
\put(16,60){\makebox(0,0)[tr]{${\scriptstyle r_2}$}}
\put(28,48){\makebox(0,0)[bl]{$0$}}
\put(16,32){\makebox(0,0)[br]{$(1-\alpha_1)(\alpha_2-\alpha_1^2)$}}
\put(20,16){\makebox(0,0)[bl]{$0$}}
\end{picture}}
\end{picture}
\]
of grading (and virtual dimension) $0$. In fact $P_K(\alpha)$ is 
isomorphic to $S_\Z^\pm(\alpha_1,\alpha_2)[\epsilon,\zeta]/I_2$, 
where $I_2$ is the ideal 
\[
\big(\,
\epsilon(\epsilon+\alpha_2-1)-(1+\alpha_1)\zeta,\;\,
\zeta(\epsilon+\alpha_1^2-1),\;\,
\zeta(\zeta-(1-\alpha_1)(\alpha_2-\alpha_1^2))\,\big)\,.
\]
As an $S_\Z^\pm(\alpha)$-module, $P_K(\alpha)$ has basis 
$\{1,\epsilon,\zeta\}$. An equivalent calculation of Anderson and 
Payne \cite[Example 1.6]{anpa:okt} interprets the latter in terms of an 
$R(T^2)$-module basis for the algebra of {\em piecewise 
exponential functions} on $\varSigma_{(1,1,2)}$.
\end{exa}

%
%
%
%
%
%
%
%
%

\section{Cohomological applications}\label{coap}

In this section we prove Theorem \ref{supmain} by translating the 
GKM-theoretic content of Proposition \ref{gkmdivwps} into the 
piecewise algebraic language of Section \ref{pial}.

Our motivation lies in the results of \cite{bafrra:ecr}, which state
that the Borel equivariant cohomology ring $H^*_{T^n}(\xvs;R)$ is
isomorphic to $P_{H\!R}(x;\varSigma)$ for any projective toric variety
(smooth or singular) whose integral cohomology is concentrated in even
dimensions. This may be thought of as a statement of compatibility
with limits, in the sense that equivariant cohomology maps the
homotopy colimit $\hocolim V\simeq X_\varSigma$ to the algebraic limit
$\lim\HR V=P_{H\!R}(x;\varSigma)$. It follows from \cite{bafrra:ecr}
that the sequence
\begin{equation}\label{forget2}
0\lra(S_R(x))\lra P_{H\!R}(x;\varSigma)\lra H^*(\xvs;R)\lra 0
\end{equation}
is short exact for any such variety. Furthermore, $P_{H\!R}(x;\varSigma)$ 
is isomorphic to the \emph{face ring} (or \emph{Stanley-Reisner algebra})
$R[x;\varSigma]$ for any \emph{smooth} fan. 

Working over a field immediately simplifies the situation; for example,  
\begin{equation}\label{ratface}
H^*_{T^n}(X_\varSigma;\Q)\;\cong\;P_{H\Q}(x;\varSigma)\;\cong\;
\Q[x;\varSigma]
\end{equation}
holds for \emph{any} fan $\varSigma$.

Kawasaki's calculations \cite{kaw:ctp} confirm that \eqref{forget2} is
short exact for every weighted projective space.  In other words,
$P_{H\!R}(x;\vsc)$ reduces to the face ring whenever $\chi_j=1$ for
every $0\leq j\leq n$.

We now explain how to interpret the GKM description of Proposition
\ref{gkmdivwps} as the limit of an appropriate contravariant $\vsc$ 
diagram. To proceed, we must therefore identify $\vsc$ more explicitly. 
For general weights $\chi$, this may be difficult; in the divisive case, 
however, it is easy to specify the rays $r_0$, \dots, $r_n$ precisely.
Bearing in mind that $\chi$ is normalised \eqref{norm}, we set
\begin{equation}\label{raysvsc}
\begin{pmatrix}
r_0&\!\!\dots\!\!&r_n
\end{pmatrix}
\;\;=\;\;
\begin{pmatrix}
-1&1&0&\dots&0\\
-\chi_2&0&1&\dots&0\\
\vdots&\vdots&\vdots&\ddots&\vdots\\
-\chi_n&0&0&\dots&1
\end{pmatrix}
\end{equation}
as an $n\times(n+1)$ matrix. The cones $\sigma_A$ of $\vsc$ are
generated by rays $\{r_i:i\notin A\}$, as $A$ ranges over all
non-empty strictly increasing subsequences $a_1,\dots,a_d$ of
$0,\dots,n$. In particular, the $n$--dimensional cones are $\sigma_0$,
\dots, $\sigma_n$, and $\sigma_A\cap\sigma_{A'}=\sigma_{A\cdot A'}$ 
holds for any $A$ and $A'$, where $A\?\cdot\? A'$ is given by 
juxtaposition and reordering.

In order to study the diagrams $KV$, $\MUV$ and $HV$
of Section \ref{pial}, we must first identify the linear forms of 
\eqref{linforms} for $\vsc$.

For every $0\leq k<l\leq n$, the $(n-1)$--dimensional cone
$\sigma_{k,l}$ is generated by the columns of the
$n\times(n-1)$-matrix obtained from \eqref{raysvsc} by deleting
columns $k$ and $l$. So a basis for $\R_{\sigma^\perp_{k,l}}$
consists of a single primitive integral vector $w$, orthogonal to all
remaining columns. If $1\leq k$, then
\[
w\;=\;(0,\dots,0,-\chi_l/\chi_k,0,\dots,0,1,0,\dots,0)
\]
(non-zero in positions $k$ and $l$) satisfies the conditions; if 
$k=0$, then $w=(0,\dots,0,1,0\dots,0)$ suffices.

With reference to Examples \ref{defplp}, \ref{defbf}, and \ref{defpp}, 
we may now deduce the following.
\begin{lem}\label{idealsvsc}
For any cone $\sigma_{k,l}$ and any $0\leq k<l\leq n$ in $\vsc$, the
principal ideals $J_{k,l}\mathbin{\emph{$\letbe$}} J_{\sigma_{k,l}}$
in $S_\Z^\pm(\alpha)$, $\MU_{T^n}^*$, and $S_\Z(x)$ are generated by
\smash{$1-\alpha_l\alpha_k^{-\chi_l/\chi_k}$,
  $e\big(\alpha_l\alpha_k^{-\chi_l/\chi_k}\big)$, and
  $x_l-(\chi_l/\chi_k)x_k$} respectively, where $x_0=0$ and
$\alpha_0=1$. \qed
\end{lem}

Lemma \ref{idealsvsc} summarizes the input required to prove Theorem
\ref{supmain}. 

\begin{thm}\label{supmain}
For any divisive weighted projective space $\Pc$:
\begin{enumerate}
\item[{\bf 1.}]
$K_{T^n}^*(\Pc)$ is isomorphic as an $S^\pm_{K^*}(\alpha)$-algebra to
$P_K(\alpha;\vsc)$;
\item[{\bf 2.}]
$\MU_{T^n}^*(\Pc)$ is isomorphic as an $MU_{T^n}^*$-algebra to
$P_{M\1U}(\vsc)$;
\item[{\bf 3.}]
$H_{T^n}^*(\Pc;R)$ is isomorphic as an $S_R(x)$-algebra to 
$P_{H\2R}(x;\vsc)$.
\end{enumerate}
\end{thm}
\begin{proof}
We give the details for {\bf 1}.

Invoking \eqref{etdij} and Proposition \ref{gkmdivwps}, we must
identify the algebra $\varGamma_{\,\bP(\chi)}$ with
$\lim_{\scat{gcalg}}KV$. The former is given by
\[
\left\{g=g(\alpha):
\text{$1-\alpha_{n-j}\alpha_{n-i}^{-\chi_{n-j}/\chi_{n-i}}$ 
divides $g_i-g_j$ for all $0\leq j<i\leq n$}\right\}\;\leq\;
\prod_iS^\pm_{K^*}(\alpha),
\]
and the universal properties of the latter suggest that we proceed by 
finding compatible homomorphisms 
$h_{a_0,\dots,a_d}\colon\varGamma_{\,\bP(\chi)}\ra KV(\sigma_{a_0,\dots,a_d})$
for every cone $\sigma_{a_0,\dots,a_d}$ in $\vsc$. It follows from Example 
\ref{defplp} and Lemma \ref{idealsvsc} that
\[
KV(\sigma_{a_0,\dots,a_d})\;=\;
S^\pm_{K^*}(\alpha)/J_{a_0,\dots,a_d},
\]
where $J_{a_0,\dots,a_d}$ denotes the ideal generated by the Laurent
polynomials $1-\alpha_l\alpha_k^{-\chi_l/\chi_k}$ as $k,l$ 
ranges over the length $2$ subsequences of $0\leq a_0,\dots,a_n\leq n$.

Given any $g$ in $\varGamma_{\,\bP(\chi)}$, we first consider cones of
dimension $n$, and define $h_k(g)\letbe g_{n-k}$ in 
$S^\pm_{K^*}(\alpha)$ for every $0\leq k\leq n$. On cones of dimension 
$n-1$, we let
\[
h_{k,l}(g)\;\letbe\; g_{n-k}\;\equiv\;
g_{n-l}\mod\; 1-\alpha_l\alpha_k^{-\chi_l/\chi_k}
\]
in $KV(\sigma_{k,l})=
S^\pm_{K^*}(\alpha)/(1-\alpha_l\alpha_k^{-\chi_l/\chi_k})$, for 
every $0\leq k<l\leq n$. This is well-defined, because 
$1-\alpha_l\alpha_k^{-\chi_l/\chi_k}$ divides 
$g_{n-k}-g_{n-l}$ in $S^\pm_{K^*}(\alpha)$. The definition extends to 
\[
h_{a_0,\dots,a_d}(g_i)\;\letbe\;g_{n-a_0}\;\equiv\;\cdots\;\equiv
\;g_{n-a_d}\mod\;J_{a_0,\dots,a_d}
\]
for any $2\leq d\leq n$, because the $g_{n-a_0}$, \dots, $g_{n-a_d}$ 
satisfy precisely the required pairwise divisibility conditions in 
$S^\pm_{K^*}(\alpha)$. Moreover, $h_{a_0,\dots,a_d}$ is a homomorphism of 
$S^\pm_{K^*}(\alpha)$-algebras, by definition.

In order to confirm the compatibility of the $h_A$ over
$\cat{cat}^{op}(\vsc)$, we note that every morphism takes the form
$\sigma_A\supseteq\sigma_{A\cdot A'}$. The corresponding projection
$\smash{q_{\sigma_A,\sigma_{A\cdot A'}}}\colon KV(\sigma_A)\ra KV(\sigma_{A\cdot A'})$
is induced by the inclusion $J_A\leq J_{A\cdot A'}$, which adjoins
the expressions 
$1-\alpha_l\alpha_k^{{}}{\mskip-7mu}^{-\chi_l/\chi_k}$ as $k,l$
ranges over the length $2$ subsequences of $A'$: so compatibility is 
assured. We have therefore constructed a homomorphism 
$h\colon\varGamma_{\,\bP(\chi)}\to\lim_{\scat{gcalg}}KV$ of 
$S^\pm_{K^*}(\alpha)$-algebras. 

We conclude by showing that $h$ is automatically an isomorphism. 
Given distinct elements $g$ and $g'$ of $\varGamma_{\,\bP(\chi)}$, 
there  must exist at least one $k$ such that $g_k\neq g'_k$ as 
elements of $S^\pm_{K^*}(\alpha)$; hence $h_k(g)\neq h_k(g')$ in 
$KV(\sigma_k)$, and $h$ is monic. Similarly, any element 
$(g_A)$ of $\lim_{\scat{gcalg}}KV$ determines $(g_a)$ in 
$\varGamma_{\,\bP(\chi)}$, by restricting to $n$-dimensional cones;
thus $h(g_a)=(g_A)$, and $h$ is epic.

The entire argument applies to {\bf 2} and {\bf 3}, with
minor modifications. For $H_{T^n}^*(\Pc;R)$, the statement is also 
a special case of \cite[Proposition 2.2]{bafrra:ecr}.
\end{proof}

Informally, the connection between $T^n$-equivariant bundles over 
$\xvs$ and piecewise Laurent polynomials on $\varSigma$ is easy to 
make.  Every such bundle determines a representation of $T^n$ on 
the fibre at each fixed point, and therefore on each maximal cone. 
These representations must be compatible over any 
$T^{n-1}$-invariant $S^2$ containing two fixed points, and therefore 
on cones of codimension $1$.

The close relationship between GKM theory and piecewise algebra has
long been known over fields of characteristic zero, and Theorem
\ref{supmain} is an instance of its extension to integral situations. 

%
%
%
%
%
%
%
%
%
\section{Completion and Borel cohomology}\label{coboco}

In this section we introduce the completions of $K^*_{T^n}(\Pc)$ and
$MU^*_{T^n}(\Pc)$ at their augmentation ideals, and discuss their
relationships with the Borel equivariant cohomology ring of $\Pc$ 
under the Chern character and the Boardman homomorphism respectively.
We express our results in terms of certain natural transformations 
between diagrams of Section \ref{pial}.

\begin{defns}\label{comptdef}
The \emph{$K$-theory completion transformation} 
$\wedge\letbe\wedge^K\colon KV\ra\KBV$ is defined on objects by the 
ring homomorphisms
\[
S_{K^*}^\pm(\alpha)/J_\sigma\;\lra\; K^*[[\gamma]]/J_\sigma,
\]
where $\wedge(\alpha_j)=1-\gamma_j$ and
$\wedge(\alpha_j^{-1})=1+\sum_{i\geq 1}\gamma^i_j$ for 
$1\leq j\leq n$; on morphisms, $\wedge$ maps the first natural 
projection to the second. 

The \emph{cobordism completion transformation} 
$\wedge\letbe\wedge^{MU}\colon\MUV\ra\MUBV$ is defined on objects 
by the ring homomorphisms 
\[
MU^*_{T^n}/J_\sigma\;\lra\; L^*[[u]]/J_\sigma,
\]
where $\wedge(e_{T^n}(\alpha_j))=u_j$ and 
$\wedge(e_{T^n}(\alpha_j^{-1}))=[-1](u_j)$ for 
$1\leq j\leq n$; on morphisms, $\wedge$ maps the first natural 
projection to the second. 
\end{defns}

So $\wedge^K$ is induced by the homomorphism 
\begin{equation}\label{compaugk}
S_{K^*}^\pm(\alpha)\;\lra\; K^*[[\gamma]]
\end{equation}
representing completion at the augmentation ideal $I$
\cite[Chapter 10]{atma:ica}. It is well-defined because 
$\wedge(\alpha_j^{-1})=(\wedge(\alpha_j))^{-1}$ for $1\leq j\leq
n$ and $\wedge(\alpha^{w_c})=(1-\gamma)^{w_c}$ for $1\leq
c\leq n-d$, so that $\wedge$ maps $J_\sigma$ to $J_\sigma$. 
The augmentation \smash{$S_{K^*}^\pm(\alpha)\ra K^*$} 
assigns to each virtual representation its dimension.

Similarly, $\wedge^{MU}$ is induced by the homomorphism 
\begin{equation}\label{compaugmu}
MU^*_{T^n}\;\lra\; L^*[[u]]
\end{equation}
representing completion at the augmentation ideal $I$. It is 
well-defined because 
$\wedge(\alpha^{w_c})=[w_{c,1}](u_1)+_U\dots+_U[w_{c,n}](u_n)$
for $1\leq c\leq n-d$, so that $\wedge$ maps $J_\sigma$ to 
$J_\sigma$. The augmentation \smash{$MU^*_{T^n}\ra L^*$} 
forgets the $T^n$-action on each equivariant cobordism class.

\begin{defns}\label{chartdef}
The \emph{Chern transformation} $ct\colon\KBV\ra H(K\Q^*)V$ 
is given on objects by the homomorphisms
\[
K^*[[\gamma]]/J_\sigma\;\lra\; S_{K\Q^*}(x)/J_\sigma,
\]
where $ct(\gamma_j)=1-e^{zx_j}$ for $1\leq j\leq n$, and  
$ct$ embeds the scalars $K^*$ as $K^*\otimes 1$ in 
$K\Q^*\letbe K^*\otimes\Q$; on morphisms, $ct$ maps the first 
natural projection to the second. 

The \emph{Boardman transformation} 
$bt\colon\MUBV\ra H(\HwMU^*)V$ is given on objects by 
the homomorphisms
\[
L^*[[u]]/J_\sigma\;\lra\; S_{H\wedge MU^*}(x)/J_\sigma,
\]
where $bt(u_j)=\sum_{i\geq 0}b_ix_j^{i+1}$ for $1\leq j\leq n$, 
and $bt$ embeds the scalars $L^*$ in $\HwMU^*$ via the Hurewicz
homomorphism; on morphisms, $bt$ maps the first natural 
projection to the second. 
\end{defns}

So $ct$ and $bt$ are induced by the respective homomorphisms
\begin{equation}\label{charag}
K^*[[\gamma]]\;\lra\; S_{K\Q^*}(x)\sands 
L^*[[u]]\;\lra\; S_{H\wedge MU^*}(x),
\end{equation}
and are well-defined because $J_\sigma$ maps to $J_\sigma$ in each 
instance.

The commutativity of the diagrams required for the naturality of 
$\wedge$, $ct$, and $bt$ follow directly from the definitions.
They therefore induce morphisms of limits, and so define 
homomorphisms
\begin{equation}\label{cplcthomo}
\wedge^K\colon P_K(\alpha;\varSigma)\;\lra\; 
P_{K\?_B}(\gamma;\varSigma)
\sands
ct\colon P_{K_B}(\gamma;\varSigma)\;\lra\; 
P_{H(K\Q^*)}(x;\varSigma)
\end{equation}
and 
\begin{equation}\label{cplbthomo}
\wedge^{MU}\colon P_{M\1U}(\varSigma)\;\lra\; 
P_{M\1U_B}(u;\varSigma)
\sands
bt\colon P_{M\1U_B}(u;\varSigma)\;\lra\; 
P_{H(H\wedge M\1U^*)}(x;\varSigma)
\end{equation}
of piecewise structures. In particular, $\wedge^K$ and $\wedge^{MU}$
may be viewed as \emph{conewise completions}; but completion commutes 
with limits, so they coincide with the respective completions of
$P_K(\alpha;\varSigma)$ and $P_{M\1U}(\varSigma)$ at their
augmentation ideals $I$. Similarly, $ct$ and $bt$ are the 
\emph{conewise Chern} and \emph{conewise Boardman homomorphism} 
respectively.

Note that $\wedge^K$ and $\wedge^{M\1U}$ are morphisms of algebras 
over the respective completion homomorphisms \eqref{compaugk} and 
\eqref{compaugmu} of scalars. Furthermore, $ct$ and $bt$ arise from 
conewise rational isomorphisms, and are therefore rational 
isomorphisms themselves; they are also morphisms of algebras over 
the homomorphisms \eqref{charag}. In other words, \eqref{cplcthomo} 
and \eqref{cplbthomo} describe the extensions of \eqref{compaugk}, 
\eqref{compaugmu}, and \eqref{charag} to the piecewise setting.

\begin{rems}\label{comp}
The composition $ct\circ\wedge^K$ is a natural transformation
$cc\colon KV\ra H(K\Q^*)V$. It is induced by the 
homomorphism
\begin{equation}\label{ccagk}
cc\colon S_{K^*}^\pm(\alpha)\;\lra\; S_{K\Q^*}(x),
\end{equation}
which satisfies $cc(\alpha_j)=e^{zx_j}$ for all $1\leq j\leq n$. 
On limits, 
\[
cc\colon P_K(\alpha;\varSigma)\;\lra\; P_{H(K\Q^*)}(x;\varSigma)
\]
identifies $P_K(\alpha;\varSigma)$ with a subring of piecewise 
formal exponential functions (the viewpoint adopted by 
\cite{anpa:okt}, and anticipated in Example \ref{oot}). It is 
a morphism of algebras over the homomorphism \eqref{ccagk} of 
scalars.

Similarly, $bt\circ\wedge^{M\1U}$ is a natural transformation
$bc\colon\MUV\ra H(\HwMU^*)V$. It is induced by the homomorphism
\begin{equation}\label{ccagmu}
bc\colon MU^*_{T^n}\;\lra\; S_{H\wedge MU^*}(x),
\end{equation}
which satisfies $bc(e(\alpha_j))=\sum_{i\geq 0}b_ix_j^{i+1}$ for all 
$1\leq j\leq n$. On limits, 
\[
bc\colon P_{M\1U}(\varSigma)\;\lra\; 
P_{H(H\wedge MU^*)}(x;\varSigma)
\]
identifies $P_{M\1U}(\varSigma)$ with a subring of piecewise 
formal power series; it is a morphism of algebras over the 
homomorphism \eqref{ccagmu} of scalars.
\end{rems}

Using the isomorphism of Theorem \ref{supmain}.2, we may now 
interpret the homomorphisms \eqref{cplcthomo} and \eqref{cplbthomo}  
topologically. We deal first with $K$-theory.

For any $T^n\acts Y$, the Borel equivariant $K$-theory 
$K^*(ET^n\times_{T^n}Y)$ is an algebra over the coefficient ring 
$K^*(BT^n)$, which acts via the projection map $\pi$ of 
the $T^n$-bundle
\begin{equation}\label{borelproj}
Y\;\lra\; ET^n\times_{T^n}Y\;\stackrel{\pi}{\lra}\; BT^n.
\end{equation}
Atiyah and Segal define a homomorphism $\lambda\colon
K_{T^n}^*(Y)\to K^*(ET^n\times_{T^n}Y)$ by assigning the vector
bundle $(ET^n\times\theta)/T^n$ to each $T^n$-equivariant vector
bundle $\theta$ over $Y$. For compact spaces such as $\xvs$, they
prove \cite[Theorem 2.1]{AtiSeg69} that $\lambda$ is completion at 
the augmentation ideal $I$. If, for example, $Y$ is the $1$-point
$T^n$-space $*$, then $\lambda\colon K_{T^n}^*\to K^*(BT^n)$
corresponds to the completion map \eqref{compaugk}, and identifies 
the coefficients $K^*(BT^n)$ with $K^*[[\gamma]]$, as in Example
\ref{defbktd}. Furthermore, $\gamma_j$ is the $K$-theoretic Euler 
class of the $j$th canonical line bundle over $BT^n$, for 
$1\leq j\leq n$. In general,  $\lambda$ may be interpreted as 
converting the $S_{K^*}^\pm(\alpha)$-algebra structure of 
$K_{T^n}^*(Y)$ to the $K^*[[\gamma]]$-algebra structure of 
$K^*(ET^n\times_{T^n}Y)$. 

The \emph{Chern character} $ch\colon K^*(ET^n\times_{T^n}Y)\to
H^*(ET^n\times_{T^n}Y;K\Q^*)\belet H_{T^n}^*(Y;K\Q^*)$ 
is the natural transformation of cohomology theories induced by the 
Hurewicz morphism
\begin{equation}\label{defch}
K\;\simeq\; S^0\!\wedge\? K\;\stackrel{i\wedge 1}{\lllra\;}\HwK
\end{equation}
of complex-oriented ring spectra, where $i$ denotes the unit of the
integral Eilenberg-Mac Lane spectrum $H$. On coefficient 
rings, it embeds $K^*\cong\Z[z,z^{-1}]$ in 
$\HwK^*\cong\Q[z,z^{-1}]=K\Q^*$ by $\Z<\Q$, and on 
$\C P^\infty$ it embeds $K^*(\C P^\infty)\cong K^*[[\gamma_1]]$ in 
$H^*(\C P^\infty;K\Q^*)\cong S_{K\Q^*}(x_1)$ by 
$ch(\gamma_1)=1-e^{zx_1}$. Further properties of $ch$ may be 
found in \cite[Chapter 5]{hil:GCT}, for example; in particular, it is 
always a rational isomorphism.

\begin{thm}\label{borkdwps}
For any divisive weighted projective space, $K^*(ET^n\times_{T^n}\Pc)$
is isomorphic as a $K^*[[\gamma]]$-algebra to $P_{K\?_B}(\gamma;\vsc)$;
with respect to this identification, the Atiyah-Segal completion map
$\lambda\colon~K_{T^n}^*(\Pc)\to K^*(ET^n\times_{T^n}\Pc)$ corresponds
to the conewise completion homomorphism~$\wedge^K$, and the Chern
character $ch\colon K^*(ET^n\times_{T^n}\Pc)\to H_{T^n}^*(\Pc;K\Q^*)$ 
corresponds to the conewise Chern transformation 
$ct$.
\end{thm}
\begin{proof}
Theorem \ref{supmain}.1 shows that $\lambda$ corresponds to $\wedge$,
and has target $P_{K\?_B}(\gamma;\vsc)$; so the latter is necessarily 
isomorphic to $K^*(ET^n\times_{T^n}\Pc)$ as a $K^*[[\gamma]]$-algebra.

Moreover, $ch\colon K^*(BT^n)\to H^*(BT^n;K\Q^*)$ maps $\gamma_j$ to
$1-e^{zx_j}$ for every $1\leq j\leq n$, which agrees precisely with
$ct$ as specified in \eqref{charag}. Since $ch$ is natural with
respect to the inclusion $\coprod BT^n\subset ET^n\times_{T^n}\Pc$
induced by the fixed point set, the result follows.
\end{proof}

We may extend the isomorphism of Theorem \ref{borkdwps} to any
projective toric variety $\xvs$ for which $H^*(\xvs)$ is torsion
free and concentrated in even dimensions. We interpret 
$P_{K\?_B}(\gamma;\varSigma)$ as a $K^*[[\gamma]]$-subalgebra 
of $\prod K^*[[\gamma]]$, where the latter contains one factor for each 
maximal cone of $\varSigma$, and hence for each $T^n$-fixed point.
\begin{thm}\label{borkxsig}
The Borel equivariant $K$-theory of any such $\xvs$ is isomorphic 
as a $K^*[[\gamma]]$-algebra to $P_{K\?_B}(\gamma;\varSigma)$.
\end{thm}
\begin{proof}
By \cite[Proposition 2.2]{bafrra:ecr}, $H^*_{T^n}(\xvs)$ is also
torsion free and even dimensional, and isomorphic to 
$P_H(x;\varSigma)$. So the Chern character 
$ch\colon K^*(ET^n\times_{T^n}\xvs)\to H^*_{T^n}(\xvs;K\Q^*)$ is 
monic, and identifies $K^*(ET^n\times_{T^n}\xvs)$ with a subring of 
$P_{H(K\Q^*)}(x;\varSigma)$. By the naturality of $ch$, the inclusion 
of the fixed point set and the projection of \eqref{borelproj} 
induce a commutative diagram
\begin{equation}\label{chcd}
\begin{CD}
\prod K^*[[\gamma]]@<K^*(\iota)<<K^*(ET^n\times_{T^n}\xvs)\!
@<{K^*(\pi)}<<\hspace{-10pt}K^*[[\gamma]]\\
@V{ch}VV@V{ch}VV@VV{ch}V\\
\prod S\2_{K\Q^*}(x)@<<H^*(\iota)<P_{H(K\Q^*)}(x;\varSigma)
@<<H^*(\pi)<S\2_{K\Q^*}(x)\quad
\end{CD}\quad,
\end{equation}
in which $H^*(\iota)$, $H^*(\pi)$, $K^*(\pi)$, and all maps $ch$, 
are monic. It follows that $K^*(\iota)$ is also monic, and that 
we may identify the elements $\gamma_j$ in 
$K^*(ET^n\times_{T^n}\xvs)$, for $1\leq j\leq n$. The image of 
$K^*(\iota)$ automatically lies in the subalgebra 
$P_{K\!_B}(\gamma;\varSigma)<\prod K^*[[\gamma]]$, because $\iota$ 
factors through the equivariant $1$-skeleton of $\xvs$; so it 
remains to show that the image is the entire subalgebra.

Note that the augmentation ideal $I$ of $K^*[[\gamma]]$ is 
generated by the $\gamma_j$, and its image $ch(I)=(x)$ is 
generated by the $x_j$. Furthermore, the filtration by powers 
of $I$ coincides with the skeletal filtration for $K^*(BT^n)$.

Choose $f\letbe f(\gamma)$ in $P_{K\?_B}(\gamma;\varSigma)$, and 
assume that its augmentation is zero without loss of generality. 
So $ch(f)=f(1-e^{zx})$ is a piecewise formal power series in 
$\prod S_{K\Q^*}(x)$, and may be rewritten as a piecewise 
polynomial in the variables $x$ over $K\Q^*$. As such, it takes 
the form $H^*(\iota)(f')$ for some element $f'= f'(x)$ in 
$P_{H(K\Q^*)}(x;\Sigma)$. Since $f'(0)=0$, it has filtration 
$q_1\geq 1$ with respect to $(x)$; so the integrality properties 
of $ch$ ensure the existence of an element $f_1=f_1(x)$ in 
$K^*(ET^n\times_{T^n}\xvs)$, such that $ch(f_1)\equiv f'$ mod 
$I^{q_1+1}$. Hence $K^*(\iota)(f_1)\equiv f$ mod $I^{q_1+1}$ 
in $P_{K\?_B}(\gamma;\varSigma)$.

Now iterate this procedure on $f-K^*(\iota)(f_1)$, to obtain a 
sequence of elements $(f_n=f_n(x))$ in $K^*(ET^n\times_{T^n}\xvs)$ 
for which
\[
K^*(\iota)(f_n)\;\equiv\;f-K^*(\iota)(f_1+\dots+f_{n-1})
\;\;\text{mod}\;\; I^{q_n+1},
\]
with $q_1<\dots<q_n$. Since $K^*(ET^n\times_{T^n}\xvs)$ and 
$P_{K\?_B}(\gamma;\varSigma)$ are $I$-adically complete, it 
follows that $f_1+\dots+f_n+\cdots$ converges to an element 
$f^\circ$ in the former, and that $K^*(\iota)(f^\circ)=f$ in 
the latter. So $K^*(\iota)$ is epic, as required.
\end{proof}

We now turn to the cobordism versions of our previous two
results.

In \cite{td:bgm}, tom Dieck introduces the \emph{bundling transformation}
 $\alpha\colon MU^*_{T^n}(Y)\ra MU^*(ET^n\times_{T^n}Y)$, 
which is proven in \cite{coma:ctc} to be completion at the augmentation 
ideal. If $Y=*$, then $\alpha$ reduces to the homomorphism 
\eqref{compaugmu}, and identifies $MU^*(BT^n)$ with $MU^*[[u]]$, as in
Example \ref{defbcd}. Each $u_j$ is the cobordism Euler class of the
$j$th canonical line bundle over $BT^n$, for $1\leq j\leq n$.

The Boardman homomorphism $bh\colon MU^*(ET^n\times_{T^n}Y)\to
H^*_{T^n}(Y;\HwMU^*)$ is induced by the Hurewicz morphism
\[
MU\;\simeq\;S^0\?\wedge MU\;
\stackrel{i\wedge 1}{\llra}\;\HwMU,
\]
by analogy with \eqref{defch}. On coefficient rings it embeds $L^*$ in 
$\HwMU^*$ by the Hurewicz homomorphism, and on $\C P^\infty$ it 
embeds $MU^*(\C P^\infty)\cong L^*[[u_1]]$ in 
$\HwMU^*(\C P^\infty)\cong H_*(MU)[[x_1]]$  
by $bh(u_1)=\sum_{i\geq 0}b_ix_1^{i+1}$. Further properties of $bh$ 
may be found in \cite{ada:shg}, for example.

We are now in a position to state the cobordism versions of 
Theorems \ref{borkdwps} and \ref{borkxsig}; they are verified by 
substituting $bh$ for $ch$ in each of the proofs.  
\begin{thm}\label{bormudwps}
For any divisive weighted projective space, $\MU^*(ET^n\times_{T^n}\Pc)$
is isomorphic as an $L^*[[u]]$-algebra to $P_{MU\!_B}(u;\varSigma)$; with 
respect to this identification, the bundling transformation
$\alpha\colon\MU_{T^n}^*(\Pc)\to\MU^*(ET^n\times_{T^n}\Pc)$ corresponds
to the conewise completion homomorphism $\wedge^{MU}$, and the Boardman 
homomorphism 
$bh\colon \MU^*(ET^n\times_{T^n}\Pc)\to H_{T^n}^*(\Pc;\HwMU^*)$ 
corresponds to the conewise Boardman transformation $bt$. 
\end{thm}
Theorem \ref{bormuxsig} applies to projective toric varieties $\xvs$ 
whose integral cohomology is free and even.
\begin{thm}\label{bormuxsig}
The Borel equivariant $\MU$-theory of any such $\xvs$ is isomorphic 
as an $L^*[[u]]$-algebra to $P_{M\1U\?_B}(u;\varSigma)$.
\end{thm}
For the proof, we interpret $P_{M\1U\?_B}(u;\varSigma)$ as an 
$L^*[[u]]$-subalgebra of $\prod L^*[[u]]$. 

Note that Theorems \ref{borkxsig} and \ref{bormuxsig} apply to all
smooth toric varieties, and to a wider class of singular examples than
Theorem \ref{supmain}; in particular, they hold for iterated products
of arbitrary weighted projective spaces.  They therefore provide
evidence for the conjecture that $K_{T^n}^*(\xvs)$ and
$MU_{T^n}^*(\xvs)$ are isomorphic to $P_K(\alpha;\varSigma)$ and
$P_{M\1U}(\varSigma)$ respectively, for \emph{any} projective toric
variety whose integral cohomology is free and even. Without further
proof, however, the most we can claim is that each pair of algebras
share the same completion.

Combining Theorem \ref{bormuxsig} with \cite[Theorem 6.4]{krum:acr}
confirms that the equivariant algebraic cobordism ring of a smooth
projective toric variety $Y$ is isomorphic to
$MU^*(ET^n\times_{T^n}Y)$. This fact also follows from 
\cite[Theorem 3.7]{kikr:ecf}. Most recently, Gonzalez and Karu have 
defined \emph{operational} equivariant algebraic cobordism, and 
\cite[Theorem 7.2]{goka:bac} proves that their ring is isomorphic to 
$MU^*(ET^n\times_{T^n}Y)$ for any quasiprojective toric variety $Y$ to 
which Theorem \ref{bormuxsig} applies, singular or otherwise. 
No analogous coincidences arise in Sections \ref{coap} or 
\ref{smca}, because the coefficient rings $L^*[[u]]$ of the algebraic 
theories are complete. 

%
%
%
%
%
%
%
%

\section{The smooth case}\label{smca}

A version of Theorem \ref{supmain} for \emph{smooth} fans may well be
known to experts, but statements are difficult to find in the
literature.  There are, however, explicit references such as
\cite{vevi:hak} and \cite{anpa:okt} to analogous results in
equivariant \emph{algebraic} and \emph{operational} $K$-theory
respectively. The goal of this section is to outline a proof (and
certain consequences) of Theorem \ref{supmain} for smooth polytopal
fans $\varSigma$, thereby confirming that all three forms of
$K$-theory agree on the corresponding $\xvs$. In this context,
\emph{polytopal} indicates that $\varSigma$ is the normal fan of a
compact simple polytope $P_\varSigma$, and therefore is
\emph{complete}. We expect that our results may be extended to more
general fans by applying the methods of \cite{fra:dtv}.

Our proof relies on the work of Harada and Landweber \cite{HaL}, which
deals with symplectic manifolds $(M,\omega)$ equipped with a
Hamiltonian action of $T^n$. In \cite[Definition 4.1]{HaL}, such an
action is defined to be \emph{GKM} whenever the fixed point set is
finite and the isotropy weights at each fixed point $p$ are pairwise
linearly independent. The latter requirement may be restated in the
notation of Section \ref{gegkmth} by decomposing the $n$-dimensional
representation $\rho_p$ of $T^n$ on $T_p(M)$ as a sum
$\bigoplus_{j=1}^n\alpha^{J_{p,j}}$ of $1$-dimensionals, and insisting
that the $J_{p,j}$ are pairwise linearly independent in $\Z^n$ for
each fixed point $p$.
 
Let $\varSigma$ be smooth and polytopal, so that $\xvs$ is smooth, 
compact and projective. In fact, $\xvs$ is also a 
symplectic toric manifold in the sense of \cite[Part XI]{cds:lsg},
with respect to the induced symplectic structure arising 
from its projective embedding. The associated  moment
map $\varPhi$ can be identified with the orbit map 
$\xvs \to \xvs/T^n \cong P_\varSigma$; moreover, $P_\varSigma$
is a \emph{Delzant polytope} \cite{del:hpi, cds:lsg}. The fixed 
points of the action are precisely the inverse images of the 
vertices of $P_\varSigma$; in particular, the set of fixed points 
is finite. Furthermore, since $P_\varSigma$ is Delzant, the
edges incident on any vertex $t$ of $P_\varSigma$ are specified by 
$n$ primitive integral vectors $J_{t,s_1}, \dots, J_{t,s_n}$, which 
form a basis for the standard lattice $\Z^n$ in the ambient $\R^n$. 
Then $T^n\acts\xvs$ is given in a neighborhood of $t$ by the
representation $\alpha_t\letbe\bigoplus_{j=1}^n\alpha^{J_{t,s_j}}$, 
and the action is GKM. 

We may now apply \cite[Theorem 4.4]{HaL} to $T^n\acts\xvs$, by 
noting that $\varPhi$ has compact domain, so every component 
is proper and bounded below.
\begin{thm}\label{hal44}
For any smooth polytopal fan $\varSigma$, the inclusion of the fixed 
point set induces an isomorphism of $K_{T^n}^*(\xvs)$ with
\[
\varGamma_\varSigma\;=\;\left\{y:\text{$1-\alpha^{J_{t,s}}$ divides
$y_s-y_t$ for all $s\prec t$}\right\}\;\leq\;\prod_t
K^*_{T^n},
\]
where $t$ ranges over the vertices of $P_\varSigma$.
\end{thm}

Theorem \ref{hal44} also holds for $\MU_{T^n}^*(\xvs)$, following our
observations on $\MU_{T^n}\,$-Euler classes in Section \ref{gegkmth}.
\begin{cor}\label{cormain}
For any smooth polytopal fan $\varSigma$:
\begin{enumerate}
\item[{\bf 1.}]
$K_{T^n}^*(\xvs)$ is isomorphic as an $S^\pm_{K^*}(\alpha)$-algebra to
$P_K(\alpha;\varSigma)$;
\item[{\bf 2.}]
$\MU_{T^n}^*(\xvs)$ is isomorphic as an $MU_{T^n}^*$-algebra to
$P_{M\1U}(\varSigma)$;
\item[{\bf 3.}]
$H_{T^n}^*(\xvs;R)$ is isomorphic as an $S_R(x)$-algebra to 
$P_{H\!R}(x;\varSigma)$.
\end{enumerate}
\end{cor}

Corollary \ref{cormain} follows from Theorem \ref{hal44} by adapting 
the proof of Theorem \ref{supmain}.

Since $\varSigma$ is smooth, each algebra of Corollary \ref{cormain}
admits an alternative description in terms of the face ring
$R[\varSigma]$, inspired by the isomorphism $H_{T^n}^*(\xvs;R)\cong
R[\varSigma]$ mentioned in Section \ref{coap}. The face ring
associates $2$--dimensional indeterminates $y_j$ to the the rays $v_j$
of $\varSigma$ for $1\leq j\leq m$, and may first have been described
as a limit in \cite{paravo:csa}.

Given any set $\omega$ of rays, it is convenient to denote the set of 
variables $\{y_j:v_j\in\omega\}$ by $y(\omega)$, and to abbreviate 
the monomial $\prod_{y(\omega)}y_j$ to $y_\omega$.
\begin{defn}\label{defsr}
  The \emph{face diagram} $S\2_R\colon\catsop\ra\catgca$ has
\begin{equation}\label{srsig}
  S\2_R(\sigma)\;=\;S\2_R(y(\sigma))\sands
  S\2_R(p_{\tau,\sigma})\;=\;f_{\tau,\sigma}\,,
\end{equation}
where $f_{\tau,\sigma}\colon S\2_R(\tau)\ra S\2_R(\sigma)$ is induced 
by annihilating those $y_i$ for which $v_i\in\tau\setminus\sigma$.
\end{defn}
The face ring $R[\varSigma]=R[y;\varSigma]$ is the limit of $S\2_R$,
and is additively generated by those monomials $y^J\letbe
y_1^{j_1}\cdots y_m^{j_m}$ whose support $\prod_{j_k\neq 0}y_k$ is
$y_\sigma$ for some cone $\sigma$, where $J=(j_1,\dots,j_m)$ lies in
$\Z_{\geq}^m$. There is therefore a canonical isomorphism
\[
S\2_R(y)/(y_\omega:\omega\notin\varSigma)\;
\stackrel{\cong}{\lra}\;R[y;\varSigma] 
\] 
of graded $S\2_R(y)$-algebras. 

The version required for $K_{T^n}^*(\xvs)$ involves $0$--dimensional 
indeterminates $\beta_j$, for $1\leq j\leq m$. 
\begin{defn}\label{deflf}
The \emph{Laurent face diagram} $S_{K^*}^\pm\colon\catsop\ra\catgca$ has
\begin{equation}\label{laursrsig}
  S_{K^*}^\pm(\sigma)\;=\;S_{K^*}^\pm(\beta(\sigma))\sands
  S_{K^*}^\pm(p_{\tau,\sigma})\;=\;f_{\tau,\sigma}\,,
\end{equation}
where $f_{\tau,\sigma}\colon S_{K^*}^\pm(\tau)\ra S_{K^*}^\pm(\sigma)$ is 
induced by mapping those $\beta_j$ to $1$ for which 
$v_j\in\tau\setminus\sigma$.
\end{defn}
In this case the \emph{Laurent face algebra} 
$F_K[\varSigma]=F_K[\beta;\varSigma]$ is the limit of 
$S_{K^*}^\pm$, 
and is additively generated by those monomials $(1-\beta)^J\letbe
(1-\beta_1)^{j_1}\cdots(1-\beta_m)^{j_m}$ whose support
$\prod_{j_k\neq 0}(1-\beta_k)$ is $(1-\beta)_\sigma$ for some cone
$\sigma$, where $J=(j_1,\dots,j_m)$ lies in $\Z^m$. There is 
therefore a canonical isomorphism
\[
S_{K^*}^\pm(\beta)/((1-\beta)_\omega:\omega\notin\varSigma)\;
\stackrel{\cong}{\lra}\;F_K[\beta;\varSigma]
\] 
of graded $S_{K^*}^\pm(\beta)$-algebras.

A similar construction is possible for the $\MU_{T^n}\,$-analogue
$F_{MU}[\varSigma]$, and for the versions involving formal power 
series used below. In the case of cohomology, $F_{H\!R}[\varSigma]$ 
coincides with the standard face ring $R[\varSigma]$, so we retain 
the latter notation.

We show that $P_K(\alpha;\varSigma)$ is isomorphic to 
$F_K[\beta;\varSigma]$ by appealing to the defining diagrams; in the 
context of algebraic $K$-theory, a proof has long been available 
\cite{vevi:hak}.

By analogy with \eqref{raysvsc}, we consider the $n\times m$ matrix 
\begin{equation}\label{vstardef}
\xi\;=\;\xi_\varSigma\;\letbe\;
\begin{pmatrix}
v_1&\!\!\cdots\!\!&v_m
\end{pmatrix}.
\end{equation}
This notation is consistent with Remark \ref{functorial}, for we may 
view $\xi$ as a map $\varSigma'\to\varSigma$ of fans; the rays of 
$\varSigma'$ are the standard basis vectors in $\R^m$, and its cones 
$\sigma'\letbe\{e_{i_1},\dots,e_{i_k}\}$ correspond bijectively to the 
cones $\sigma\letbe\{v_{i_1},\dots,v_{i_k}\}$ of $\varSigma$. For any 
$n$--dimensional $\sigma$ we write $\xi_\sigma$ for the $n\times n$ 
submatrix of $\xi$ whose columns generate $\sigma$. 
The smoothness of $\varSigma$ guarantees that every $\xi_\sigma$ is 
invertible over $\Z$. So $\xi$ defines an epimorphism 
$\Z^m\ra\Z^n$, and $\xi^{tr}$ induces monomorphisms 
$S\2_R(x)\to S\2_R(y)$ and 
$S_{K^*}^\pm(\alpha)\to S_{K^*}^\pm(\beta)$ of graded 
rings; the latter maps $\alpha^J$ to $\beta^{{\xi}^{tr}J}$ for 
any $J\in\Z^n$. 

\begin{prop}\label{KTfacering}
For any smooth polytopal fan $\varSigma$, the matrix $\xi$ induces a 
natural isomorphism 
\[
\xi^*\colon P_K(\alpha;\varSigma)\lra F_K[\beta;\varSigma]
\]
of algebras over ${\xi}^{tr}\colon S_{K^*}^\pm(\alpha)\to S_{K^*}^\pm(\beta)$.
\end{prop}
\begin{proof}
The epimorphism $\xi\colon\Z^m\ra\Z^n$ maps the generator $e_j$ to the 
ray $v_j$, for all $1\leq j\leq m$. It therefore induces an isomorphism 
$\xi^*(\sigma)\colon KV(\sigma)\to S_{K^*}^\pm(\sigma)$ 
of algebras over $\xi^{tr}$, defined by 
$\xi^*(\sigma)(\alpha^J)=\beta^{\xi(\sigma)^{tr}J}$ for any  $J\in\Z^n$. 
Moreover, 
$f_{\tau,\sigma}\cdot {\xi}^*(\tau)={\xi}^*(\sigma)\cdot q_{\tau,\sigma}$
for every morphism $\tau\supseteq\sigma$ in $\cats$, because 
$\xi_\tau\supseteq \xi_\sigma$ as submatrices of $\xi$. 

So ${\xi}^*$ is a natural isomorphism of diagrams, and induces the required
isomorphism of limits.
\end{proof}
\begin{cor}
For any smooth polytopal fan $\varSigma$, there is an isomorphism
$K^*_{T^n}(\xvs)\to F_K[\beta;\varSigma]$ of algebras over
$\xi^{tr}$.
\end{cor}
\begin{proof}
Combine Corollary \ref{cormain} with Proposition \ref{KTfacering}.
\end{proof}
\begin{exa}
If $\xi$ is the $n\times(n+1)$ matrix $(\om -1\,\;I_n\om)$, then 
$X_\varSigma$ is $\CPn$ and there is an isomorphism 
\[
K^*_{T^n}(\CPn)\;\llra\;
S_{K^*}^\pm(\beta_0,\dots,\beta_n)
\om\Big/\big(\prod_{j=0}^n(1-\beta_j)\om\big) 
\]
of algebras over ${\xi}^{tr}$. An equivalent formula appears, for example, 
in \cite{grwi:pdk}.
\end{exa}

Finally, for any smooth polytopal fan $\varSigma$, we describe the
Borel equivariant $K$-theory of $\xvs$ in terms of the face ring
$F_K[[\delta;\varSigma]]$, whose indeterminates are
$0$--dimensional. This is an algebra over $K^*[[\delta]]$, and is the
limit of the $\varSigma^{op}$-diagram that assigns
$K^*[[\delta_1,\dots,\delta_m]]/(\delta_\omega:\omega\notin\sigma)$ to
each cone $\sigma\in\varSigma$.  It is also the completion of
$F_K[\beta;\varSigma]$ at the augmentation ideal
$(1-\beta_1,\dots,1-\beta_m)$, where $\delta_j=1-\beta_j$ for all
$1\leq j\leq m$.

We retain the notation of \eqref{vstardef}, observing that ${\xi}^{tr}$ of 
Proposition \ref{KTfacering} extends to a homomorphism 
$\xi^{tr}\colon K^*[[\gamma]]\to K^*[[\delta]]$. 
\begin{prop}
For any smooth polytopal fan $\varSigma$, there are isomorphisms
\[
K^*(ET^n\times_{T^n}\xvs)\lra
P_{K\?_B}(\gamma;\varSigma)
\stackrel{\xi^*}{\lra}F_K[[\delta;\varSigma]]
\]
of algebras over $K^*[[\gamma]]$ and $\xi^{tr}$ respectively.
\end{prop}
\begin{proof}
The first isomorphism is the completion of Corollary \ref{cormain}.2, and
the second is the completion of Proposition \ref{KTfacering}.  
\end{proof}

%
%
%
%
%
%
%
%
%

\end{document}